\documentclass{eajam}
\usepackage{charter}
\usepackage[charter]{mathdesign}
\usepackage{bm}              

\def\Ec{{\mathbb E}}
\def\Pc{{\mathbb P}}

\def\norm#1{\|#1\|}

\def\bF #1{\| #1 \|_F}   
\def\BF #1{\| #1 \|_F^2}  

\allowdisplaybreaks[1]  

\usepackage{graphicx}
\usepackage{subfigure}
\usepackage{epstopdf} 

\usepackage{threeparttable} 

\usepackage{enumerate}  

\usepackage{algorithmicx,algorithm, algpseudocode}

\usepackage{multirow}

\usepackage{color,xcolor}

\numberwithin{property}{section}
\numberwithin{figure}{section}
\numberwithin{table}{section}

\makeatletter
\renewcommand{\section}{\@startsection{section}{1}{0mm}
{-\baselineskip}{0.5\baselineskip}{\Large\bf\leftline}}
\makeatother

\begin{document}

\title{On the relaxed greedy randomized Kaczmarz methods with momentum acceleration for solving matrix equation $AXB=C$}

\author{Nian-Ci Wu \affil{1}, Yang Zhou \affil{2}, and Zhaolu Tian \affil{3}\comma\corrauth}

\address{
\affilnum{1}\ School of Mathematics and Statistics, South-Central Minzu University, Wuhan 430074, China.\\
\affilnum{2}\ School of Mathematics and Computational Science, Huaihua University, Huaihua 418000, China.\\
\affilnum{3}\ College of Applied Mathematics, Shanxi University of Finance and Economics, Taiyuan 030006, China.\\
        }
\markboth{N.-C. Wu, Y. Zhou, and Z. Tian}{The PmRGRK and NmRGRK methods}

\email{{\tt tianzhaolu2004@126.com} (Z. Tian)}

\begin{abstract}
With the growth of data, it is more important than ever to develop an efficient and robust method for solving the consistent matrix equation $AXB=C$. The randomized Kaczmarz (RK) method has received a lot of attention because of its computational efficiency and low memory footprint. A recently proposed approach is the matrix equation relaxed greedy RK (ME-RGRK) method, which greedily uses the loss of the index pair as a threshold to detect and avoid projecting the working rows onto that are too far from the current iterate. In this work, we utilize the Polyak's  and Nesterov's  momentums to further speed up the convergence rate of the ME-RGRK method. The resulting methods are shown to converge linearly to a least-squares solution with minimum Frobenius norm. Finally, some numerical experiments are provided to illustrate the feasibility and effectiveness of our proposed methods. In addition, a real-world application, i.e., tensor product surface fitting in computer-aided geometry design, has also been presented for explanatory purpose.
\end{abstract}


\maketitle


\section{Introduction}\label{sec1}
Consider an iterative solution of large-scale linear matrix equation of the form
\begin{equation}\label{AXB=C}
AXB=C.
\end{equation}
That is,  $A \in \mathbb{R}^{m\times n}$ and $B \in \mathbb{R}^{n\times p}$ ($m,p\geq n$) are two coefficient matrices, $C \in \mathbb{R}^{m\times p}$ is a right-hand side,
 and $X \in \mathbb{R}^{n\times n}$ is an unknown matrix. This kind of problem  has been discussed in the areas of  a variety of real-world applications, such as tensor product surface fitting in computer-aided geometry design; see, e.g., \cite{20LLFZ}.


For solving such a problem \eqref{AXB=C}, the gradient-type method is a highly popular representative in practice. Let $F(X)$ be a differentiable function. For $k=0,1,2,\cdots$, the gradient descent (GD) iteration can be formulated as
\begin{align}
  X^{(k+1)} = X^{(k)} - \alpha \nabla F(X^{(k)})
\end{align}
with $\alpha$ being a step-size and $\nabla F(X)$ being the gradient of $F(X)$. We can see that $X^{(k+1)}$ follows the negative gradient of $F(X)$ to locate its minimum value. If one takes $F(X) = \bF{C-AXB}^2/2$, the gradient-based iterative method \cite[Theorem 2]{08DLD} emerges. Polyak's  momentum, popularly known as heavy ball momentum, is one of the most influential acceleration procedures for GD to solve the unconstrained minimization problems \cite{64Polyak}. The iteration scheme is given by
\begin{align}\label{eq:Polyak's+GD}
  X^{(k+1)} = X^{(k)} - \alpha \nabla F(X^{(k)}) + \beta ( X^{(k)} - X^{(k-1)} ),
\end{align}
where $\beta\geq0$ is a momentum parameter. Nesterov's  momentum is another extension of the GD method. Specifically, given two initial matrices $X^{(0)}$, $Y^{(0)}\in \mathbb{R}^{n\times n}$, the new approximation is computed by
\begin{align}\label{eq:Nesterov's+GD}
  X^{(k+1)} = Y^{(k+1)} + \beta ( Y^{(k+1)} - Y^{(k)} )
  \quad {\rm with} \quad
  Y^{(k+1)} = X^{(k)} - \alpha \nabla F(X^{(k)}).
\end{align}
This approach was described by and named for Nesterov in \cite{83Nest}. Sutskever et al. are responsible for popularizing it in the training of neural networks with stochastic GD \cite{13SMDH}. For additional details on solving matrix equation \eqref{AXB=C}, we refer to the review \cite{16Sim} and the references, such as \cite{08DH, 22HM, 14KM, 51RM,19Zhang}.

Strohmer and Vershynin pioneered the use of randomization in the Kaczmarz iteration  \cite{Kac37} to solve an over-determined consistent linear system \cite{09SV}. 
A related and well-studied variant of this approach is to perform row selection greedily and randomly \cite{18BW1, 18BW2, 21BW}.
Very recently, by utilizing the Petrov-Galerkin conditions  \cite{Saad2003},  Wu et al. induced the Kaczmarz iteration format for matrix equation \eqref{AXB=C} and proposed the matrix equation relaxed greedy randomized Kaczmarz (ME-RGRK) method \cite{22WLZ}. Let $a_i^T$ and $b_j^T$ be the $i$th and $j$th rows of $A$ and $B^T$, respectively, and $C_{i,j}$ be the $(i,j)$th entry of $C$ for $i\in[m]$ and $j\in[p]$, where we define the set $[\ell] =\{1,2,\cdots,\ell\}$ for any integer $\ell$. After giving an initial matrix $X^{(0)}$, the Kaczmarz iteration is computed by
\begin{equation}\label{IteRK-AXB=C}
 X^{(k+1)}  = X^{(k )} + \frac{C_{i,j} - a_i^T X^{(k )} b_j}{\norm{a_{i}}^2  \norm{b_{j}}^2 } a_i   b_j^T,
\end{equation}
where the index pair $(i,j)$ is chosen according to a well-defined criterion such as an appropriate probability distribution \cite{22WLZ}. The global randomized block Kaczmarz and randomized average block Kaczmarz algorithms were presented by Niu and Zheng \cite{22NZ}. Another popular extension of the projection technique can be found in the work by Du et al. \cite{22DRS}, where a randomized block coordinate descent algorithm was given to deal with a matrix least-squares problem.

Actually, formula \eqref{IteRK-AXB=C} can be seen as a particular case of GD algorithm for minimizing the cost function
\begin{align*}
F(X) = \frac{1}{2}\frac{|C_{i,j}-a_i^T X b_j|^2}{\norm{a_i}^2  \norm{b_j}^2}.
\end{align*}
It is natural to further accelerate the convergence rate of the ME-RGRK method by utilizing the Polyak's and Nesterov's  momentum techniques. To the best of our knowledge, the momentum variant of {\it greedy randomized} iterative methods is new.

In this work, we will present the momentum variants of the ME-RGRK method to solve the matrix equation \eqref{AXB=C} and analyze their convergence. The organization of this work is as follows. We first give a brief description of the ME-RGRK method in Section \ref{sec:ME-RGRK}. Then,  the formal descriptions of the Polyak's and Nesterov's momentum variants of the ME-RGRK method are provided in Sections \ref{sec:mME-RGK+PmRGRK} and \ref{sec:mME-RGK+NmRGRK}, respectively. For simplicity, we name these two methods as PmRGRK and NmRGRK.  The corresponding convergence theories of the PmRGRK and NmRGRK methods are presented in Section \ref{sec:PmRGRK+NmRGRK_convergence}. Next,  in Section \ref{sec:PmRGRK+NmRGRK_numer}, some numerical examples are shown to demonstrate the theoretical results.
Finally, we conclude this paper with some concluding remarks and a future outlook in Section  \ref{sec:PmRGRK+NmRGRK_conclud}.

{\it Notation.} The symbol $\Ec[\cdot]$ denotes the expectation for any random variable. For any matrix $M$,  we use
$M^\dag$, ${\rm   Tr }(M)$, $\sigma_{1}(M)$, and $\sigma_{r}(M)$  to denote
the Moore-Penrose  pseudoinverse, the trace, the largest,  the smallest nonzero singular values,  respectively. The symbol $\norm{\cdot}$ is used to represent the $2$-norm of either a vector or a matrix and $\norm{\cdot}_F$ represents the Frobenius norm for a matrix. The Frobenius inner product is defined by
\begin{align*}
\langle M_1, M_2 \rangle_{F} = {\rm   Tr }(M_1^T M_2) = {\rm   Tr }(M_1 M_2^T)
\end{align*}
for any $M_1$ and $M_2$ being  with compatible dimensions.

\section{The momentum variants of ME-RGRK}\label{sec:PmRGRK+NmRGRK_algorithm}
In this section, we first briefly review the ME-RGRK method for solving the matrix equation \eqref{AXB=C}; see \cite{22WLZ}. Then, we cooperate it with the Polyak's and Nesterov's momentum techniques and present the PmRGRK and NmRGRK methods.

\subsection{The ME-RGRK method}\label{sec:ME-RGRK}

Let $X^\ast = A^{\dag}CB^{\dag}$ be a least norm least-squares solution of the linear matrix equation \eqref{AXB=C}. The next squared error in \eqref{IteRK-AXB=C} can be expressed by
\begin{align*}
  \BF{X^{(k+1)} - X^{\ast}} = \BF{X^{(k)} - X^{\ast}} - W_{i,j}(X^{(k)}),
\end{align*}
where the loss value $W_{i,j}(X^{(k)}) = |R_{i,j}^{(k)}|^2/ (\norm{a_i}^2 \norm{b_j}^2)$ with $R^{(k)}=C-AX^{(k)}B$. If
\begin{align*}
  W_{i_{k_1},j_{k_1}}(X^{(k)}) > W_{i_{k_2},j_{k_2}}(X^{(k)}),
  ~i_{k_1},i_{k_2}\in[m],~j_{k_1},j_{k_2}\in[p],
\end{align*}
we may want the index pair $(i_{k_1},j_{k_1})$ to be selected with a
larger probability prior to $(i_{k_2},j_{k_2})$, so that the larger entries of $W(X^{(k)})$ can be preferentially wiped out as far as possible.

A randomized greedy strategy in \cite{22WLZ} allows the index pair $(i_k,j_k)$ being selected such that
\begin{align*}
    W_{i_k,j_k}(X^{(k)})
       = \theta \cdot  \max\limits_{i\in [m], j\in [p]} \left\{ W_{i,j}(X^{(k)})\right\}
         +  (1-\theta)\cdot \Ec \left[  W_{i,j}(X^{(k)}) \right ],
\end{align*}
where $\theta$ is a relaxation parameter. In particular, after selecting the indices $i\in [m]$ and $j\in [p]$ with probabilities
\begin{equation*}
  \Pc({\rm Index}_1 = i)  =  \frac{\norm{a_i}^2}{\BF{A}}
  \quad {\rm and} \quad
  \Pc({\rm Index}_2 = j)  =  \frac{\norm{b_j}^2}{\BF{B}},
\end{equation*}
respectively, the expected value of the discrete random variable $(i,j)$ is given by
\begin{equation*}
 \Ec \left[  W_{i,j}(X^{(k)}) \right ]
 = \sum\limits_{i\in [m], j\in [p]}   \Pc({\rm Index}_1 = i) \cdot  \Pc({\rm Index}_2 = j) \cdot W_{i,j}(X^{(k)})
 = \frac{\BF{R^{(k)}}}{\BF{A}\BF{B}}.
\end{equation*}
This strategy can effectively detect the index pairs who have small loss.

Having the above preparations, the formal description of the ME-RGRK method is stated in Algorithm \ref{alg:ME-RGRK}. For more details, we refer to \cite[Section 4.1]{22WLZ}.

\begin{algorithm}[!htb]
\caption{The ME-RGRK method \cite{22WLZ}}
\label{alg:ME-RGRK}
\begin{algorithmic}[1]
\Require
The coefficient matrices $A \in \mathbb{R}^{m\times n}$ and $B \in \mathbb{R}^{n\times p}$ ($m,p\geq n$), and a right-hand side $C \in \mathbb{R}^{m\times p}$, an initial matrix $X^{(0)} \in \mathbb{R}^{n\times n}$, a relaxation parameter $\theta$, and the maximum iteration number $\ell$.
\Ensure
$X^{(\ell)}$.
\State {\bf for} $k=1,2,\cdots,\ell-1$ {\bf do}
\State \quad determine the index set $\Delta_k$  according to
\begin{align*}
  \Delta_k =\left\{(i,j)\Bigg|
~W_{i,j}(X^{(k)})
  \geq \delta_k \BF{R^{(k)}},~i\in[m],~j\in[p] \right\},
\end{align*}
\quad where
\begin{align*}
\delta_k =
  \frac{\theta}{\BF{R^{(k)}}}
  \max\limits_{i\in [m], j\in [p]} \left\{  W_{i,j}(X^{(k)}) \right\}
  + \frac{1-\theta}{\bF{ A }^2\bF{ B }^2};
  \end{align*}
\State \quad select the index pair $(i_k,j_k)$ from $\Delta_k$ with probability $p_{i_k,j_k}\geq 0$, where
\begin{align*}
  \sum_{(i_k,j_k)\in \Delta_k}p_{i_k,j_k}=1;
\end{align*}
\State \quad compute the next approximation as
\begin{align*}
 X^{(k+1)} = X^{(k)} + \frac{C_{i_k,j_k}-a_{i_k}^T X^{(k)} b_{j_k} }{\norm{a_{i_k}}^2  \norm{b_{j_k}}^2} a_{i_k}  b_{j_k}^T;
\end{align*}
\State {\bf endfor}
\end{algorithmic}
\end{algorithm}

\subsection{The PmRGRK method}\label{sec:mME-RGK+PmRGRK}

The mechanism behind the three-term recurrence in \eqref{eq:Polyak's+GD} involves two basic computational procedures. The first-half step updates $X^{(k)}$ along the negative gradient and the second-half step utilizes the addition of the momentum term.  The Polyak's momentum method, resulting in an accelerated convergence, is intuitive. A heavier ball will bounce less and move faster through regions of low curvature than a lighter ball due to the added momentum.

Based on this idea, we describe the calculation process of the PmRGRK method as follows. Consider the cost function
\begin{align}\label{eq:PmRGRK+NmRGRK_F_ij(X)}
F(X) = \frac{1}{2}W_{i,j}(X),~i\in[m],~j\in[p],
\end{align}
for any unknown matrix $X\in \mathbb{R}^{n \times n}$, whose gradient is easily computed by
\begin{align}\label{eq:PmRGRK+NmRGRK_gradient-F_ij(X)}
\nabla F(X)=-\frac{C_{i,j}-a_i^T X b_j }{\norm{a_i}^2  \norm{b_j}^2} a_i  b_j^T.
\end{align}
The closed-form of the Polyak's momentum variant for matrix equation Kaczmarz iteration is explicitly derived by
 \begin{align*}
 X^{(k+1)} = X^{(k)} +  \alpha \frac{C_{i,j}-a_i^T X^{(k)} b_j }{\norm{a_i}^2  \norm{b_j}^2} a_i  b_j^T
 + \beta ( X^{(k)} - X^{(k-1)} ),
\end{align*}
where $\alpha$ is a step-size and $\beta$ is a momentum parameter.

A key ingredient to guarantee fast convergence of the Kaczmarz iterative method is the construction of an appropriate criterion for the choice of row index pair $(i,j)$. Inspired by the adaptive greedy index selection strategy in the standard ME-RGRK algorithm, the PmRGRK method for solving matrix equation \eqref{AXB=C} is formally stated in Algorithm \ref{alg:PmRGRK}.

\begin{algorithm}[!htb]
\caption{The PmRGRK method}
\label{alg:PmRGRK}
\begin{algorithmic}[1]
\Require
The coefficient matrices $A \in \mathbb{R}^{m\times n}$ and $B \in \mathbb{R}^{n\times p}$ ($m,p\geq n$), and a right-hand side $C \in \mathbb{R}^{m\times p}$, two initial matrices $X^{(0)}$, $X^{(1)} \in \mathbb{R}^{n\times n}$, a relaxation parameter $\theta$, a step-size $\alpha$, a momentum parameter $\beta$, and the maximum iteration number $\ell$.
\Ensure
$X^{(\ell)}$.
\State {\bf for} $k=1,2,\cdots,\ell-1$ {\bf do}
\State \quad determine the index set $\Delta_k$  according to
\begin{align*}
  \Delta_k =\left\{(i,j)\Bigg|
~W_{i,j}(X^{(k)})
  \geq \delta_k \BF{R^{(k)}},~i\in[m],~j\in[p] \right\},
\end{align*}
\quad where
\begin{align*}
\delta_k =
  \frac{\theta}{\BF{R^{(k)}}}
  \max\limits_{i\in [m], j\in [p]} \left\{  W_{i,j}(X^{(k)}) \right\}
  + \frac{1-\theta}{\bF{ A }^2\bF{ B }^2};
  \end{align*}
\State \quad select the index pair $(i_k,j_k)$ from $\Delta_k$ with probability $p_{i_k,j_k}\geq 0$, where
\begin{align*}
  \sum_{(i_k,j_k)\in \Delta_k}p_{i_k,j_k}=1;
\end{align*}
\State \quad compute the next approximation as
\begin{align*}
 X^{(k+1)} = X^{(k)} +  \alpha \frac{C_{i_k,j_k}-a_{i_k}^T X^{(k)} b_{j_k} }{\norm{a_{i_k}}^2  \norm{b_{j_k}}^2} a_{i_k}  b_{j_k}^T
 + \beta ( X^{(k)} - X^{(k-1)} );
\end{align*}
\State {\bf endfor}
\end{algorithmic}
\end{algorithm}

\begin{remark}
Polyak's momentum has been extended to solve the constrained and distributed optimization problems, confirming its performance advantages over standard gradient-based methods; see, e.g., \cite{13GSJ,14WM}. In the context of solving liner system, Polyak's  momentum technique has been
 spurred many related works by incorporating into various randomized iterative methods, e.g., randomized coordinate descent and Kaczmarz \cite{20LR}, sketch and project \cite{20LR}, sampling Kaczmarz Motzkin \cite{22MIA}, randomized Douglas-Rachford \cite{22HSX}, doubly stochastic iterative framework \cite{22HX}, and so on.
\end{remark}

\subsection{The NmRGRK method}\label{sec:mME-RGK+NmRGRK}

We know from the physical background of the momentum method that the search may miss or overshoot the minimum value at the bottom of basins or valleys in some cases due to momentum acceleration. Nesterov's  accelerated gradient method \cite{83Nest,03Nest} is a possible remedy for it, which is a popular and effective momentum variant and closely related to the gradient descent method.


Similar to the heuristic in PmRGRK, we consider the cost function in formula \eqref{eq:PmRGRK+NmRGRK_F_ij(X)}. The NmRGRK iteration is divided into two parts. We first compute an auxiliary matrix $Y^{(k)}$ according to
\begin{align*}
Y^{(k+1)} = X^{(k)} +  \alpha \frac{C_{i,j}-a_{i}^T X^{(k)} b_{j} }{\norm{a_{i}}^2  \norm{b_{j}}^2} a_{i}  b_{j}^T
\end{align*}
for $i\in [m]$ and $j\in [p]$ with $\alpha$ being a step-size, then perform an iteration of Nesterov's  momentum, i.e.,
\begin{align*}
 X^{(k+1)} = Y^{(k+1)} + \beta ( Y^{(k+1)} - Y^{(k)} ).
\end{align*}
It indicates that $X^{(k+1)}$ is updated by using the gradients at the current iteration, as opposed to the PmRGRK method, which computes the next approximation by using the previous iterations. Algorithm \ref{alg:NmRGRK} describes the NmRGRK method in detail.

\begin{algorithm}[!htb]
\caption{The NmRGRK method}
\label{alg:NmRGRK}
\begin{algorithmic}[1]
\Require
The coefficient matrices $A \in \mathbb{R}^{m\times n}$ and $B \in \mathbb{R}^{n\times p}$ ($m,p\geq n$), and a right-hand side $C \in \mathbb{R}^{m\times p}$, two initial matrices $X^{(0)}$, $Y^{(0)} \in \mathbb{R}^{n\times n}$, a relaxation parameter $\theta$, a step-size $\alpha$, a momentum parameter $\beta$, and the maximum iteration number $\ell$.
\Ensure
$X^{(\ell)}$.
\State {\bf for} $k=0,1,2,\cdots,\ell-1$ {\bf do}
\State \quad determine the index set $\Delta_k$  according to
\begin{align*}
  \Delta_k =\left\{(i,j)\Bigg|
~W_{i,j}(X^{(k)})
  \geq \delta_k \BF{R^{(k)}},~i\in[m],~j\in[p] \right\},
\end{align*}
\quad where
\begin{align*}
\delta_k =
  \frac{\theta}{\BF{R^{(k)}}}
  \max\limits_{i\in [m], j\in [p]} \left\{  W_{i,j}(X^{(k)}) \right\}
  + \frac{1-\theta}{\bF{ A }^2\bF{ B }^2};
  \end{align*}
\State \quad select the index pair $(i_k,j_k)$ from $\Delta_k$ with probability $p_{i_k,j_k}\geq 0$, where
\begin{align*}
  \sum_{(i_k,j_k)\in \Delta_k}p_{i_k,j_k}=1;
\end{align*}
\State \quad compute the next approximation as
\begin{align*}
 X^{(k+1)} = Y^{(k+1)} + \beta ( Y^{(k+1)} - Y^{(k)} )
\end{align*}
\quad with
\begin{align*}
 Y^{(k+1)} = X^{(k)} +  \alpha \frac{C_{i_k,j_k}-a_{i_k}^T X^{(k)} b_{j_k} }{\norm{a_{i_k}}^2  \norm{b_{j_k}}^2} a_{i_k}  b_{j_k}^T;
\end{align*}
\State {\bf endfor}
\end{algorithmic}
\end{algorithm}

\begin{remark}
The updates in PmRGRK and NmRGRK are reminiscent of Polyak's  and Nesterov's  momentum techniques, respectively. They are all variants of the stochastic GD methods, which have gained much popularity due to their small memory footprint and good theoretical guarantees; see, e.g., \cite{22HNRS}. When the index pair is selected directly with a probability proportional to its Euclidean norm, we can obtain the randomized Kaczmarz method with Polyak's and Nesterov's momentums to solve the matrix equation \eqref{AXB=C}. The standard Polyak's  momentum variant of the randomized Kaczmarz method, but with no greedy selection, for solving a linear system was analyzed by Morshed et al. in \cite{20LR}. In the PmRGRK and NmRGRK methods, we propose to replace the selection of $(i_k,j_k)$ with a relaxed greedy rule. As far as we know, the introduction of Polyak's  and Nesterov's  momentums to accelerate the greedy and randomized Kaczmarz method is new.
\end{remark}

Together with increasing the number of iteration steps, the probability criterion for selecting the working rows in the coefficient matrices changes correspondingly. Therefore, the PmRGRK and NmRGRK methods are adaptive. Also, we point out that, since
\begin{align*}
    \max\limits_{i\in [m], j\in [p]} \left\{ W_{ij}(X^{(k)})\right\}
       \geq \frac{ \BF{R^{(k)}} }{\bF{ A }^2\bF{ B }^2},
\end{align*}
the index set $\Delta_k$ generated by Algorithms \ref{alg:PmRGRK} and \ref{alg:NmRGRK} will not be empty and is well defined.

\begin{remark}
When $\alpha=1$ and $\beta=0$, the PmRGRK and NmRGRK methods automatically reduce to the standard ME-RGRK method \cite{22WLZ}. The main difference between PmRGRK and ME-RGRK is the introduction of step-size $\alpha$ and momentum term $\beta ( X^{(k)} - X^{(k-1)} )$ in the computing of $X^{(k+1)}$. It is a similar story for the NmRGRK and ME-RGRK methods.
\end{remark}

\begin{remark}
Whatever the step-size is chosen at each PmRGRK and NmRGRK iteration step, we just require additional $3n^2$ flopping operations (flops) to compute the momentum term. Another computation process for both PmRGRK and NmRGRK is intensive in the selection of index pair $(i,j)$ because one needs to compute the residual entries $R_{i,j}^{(k)}$ and construct the index set $\Delta_k$. However, the index selection in the PmRGRK and NmRGRK methods uses the loss of $(i,j)$ as a threshold to detect and avoid projecting the working rows onto those that are too far from the current iteration. In most cases, the greedy iterative methods produce a higher quality of robustness and a faster convergence rate, which can outweigh the additional cost. Later, this advantage will become apparent for the test instances in the numerical section; see Section \ref{sec:PmRGRK+NmRGRK_numer}.
\end{remark}

The following part will go over several fundamental properties of the PmRGRK and NmRGRK methods.

\begin{proposition}\label{PmRGRK:gammak}
At the $k$th PmRGRK or NmRGRK iteration, let
\begin{align*}
 \zeta_k = \bF{ A }^2\bF{ B }^2 - \sum_{(i,j)\in \Omega_k} \norm{a_i}^2 \norm{b_j}^2
\end{align*}
with $\Omega_k  = \left\{ (i,j)\big|W_{ij}(X^{(k)})=0\right\}$ for $k=1,2,\cdots$. We have
\begin{align*}
\delta_k  \bF{ A }^2\bF{ B }^2 \geq \gamma_k \geq 1,
\end{align*}
where the parameter $ \gamma_k$ is defined by $ \gamma_k = \theta \bF{ A }^2\bF{ B }^2 / \zeta_k +(1-\theta)$.
\end{proposition}

\begin{proof}
An elementary computation shows that
  \begin{align*}
  \delta_k  \bF{ A }^2\bF{ B }^2
  &=  \theta
  \frac{ \bF{ A }^2\bF{ B }^2 }{ \BF{ R^{(k)} } }
  \max\limits_{i\in [m], j\in [p]} \left\{ W_{i,j}(X^{(k)}) \right\}
  + (1-\theta)  \\
   & = \theta
  \frac{\bF{ A }^2\bF{ B }^2 \max\limits_{i\in [m], j\in [p]} \left\{ W_{i,j}(X^{(k)}) \right\}}{
  \sum\limits_{i\in [m], j\in [p]}
  W_{i,j}(X^{(k)}) \norm{a_{i}}^2  \norm{b_j}^2}
  + (1-\theta) \\
   & = \theta
  \frac{ \bF{ A }^2\bF{ B }^2 \max\limits_{i\in [m], j\in [p]} \left\{ W_{i,j}(X^{(k)}) \right\}}{
  \left(
  \sum\limits_{i\in [m], j\in [p]}  -
  \sum\limits_{(i,j)\in \Omega_k}
  \right)
   W_{i,j}(X^{(k)}) \norm{a_{i}}^2  \norm{b_{j}}^2}
    + (1-\theta)  \\
   & \geq \theta
  \frac{ \bF{ A }^2\bF{ B }^2 }{
  \left(
  \sum\limits_{i\in [m], j\in [p]}  -
  \sum\limits_{(i,j)\in \Omega_k}
  \right)
   \norm{a_{i}}^2  \norm{b_{j}}^2}
    + (1-\theta),
  \end{align*}
  which yields the result in Proposition \ref{PmRGRK:gammak} immediately.  \qed
\end{proof}

\begin{proposition}
 Assume that $vec(X^{(0)})$ and $vec(X^{(1)})$ belong to the column space of $B\otimes A^T$. At the $k$th PmRGRK or NmRGRK iteration, the expectation of $W_{i,j}(X^{(k)})$ with respect to $(i,j)\in\Delta_k$ is bounded by
 \begin{align}\label{PmRGRK+eq:E_Wij(X)}
 \Ec \left[ W_{i,j}(X^{(k)}) \right] \geq \widetilde{\rho}_k \Ec \left[ \bF{ X^{(k)}-X^{\ast} }^2 \right],
\end{align}
where $\widetilde{\rho}_k =\gamma_k \widetilde{\rho}$ with $\widetilde{\rho} = \sigma_{r}^{2}(A) \sigma_{r}^{2}(B)/(\bF{ A }^2\bF{ B }^2)$.
\end{proposition}

\begin{proof}
 By induction, it holds true for $vec(X^{(k)} - X^{\ast})$ being in the column space of $B\otimes A^T$. The expectation of $W_{i,j}(X^{(k)})$ with respect to $(i,j)\in\Delta_k$ conditioned on the first $k$ iterations in the PmRGRK or NmRGRK method is given by
\begin{align*}
\Ec_k \left[  W_{i,j}(X^{(k)}) \right]
  & = \sum_{(i,j)\in \Delta_k } p_{i,j}
  W_{i,j}(X^{(k)}) \notag \\
  & \geq \delta_k  \BF{R^{(k)}} \sum_{(i,j)\in \Delta_k } p_{i,j} \\
  & = \delta_k  \BF{R^{(k)}} = \delta_k  \BF{ A (X^{(k)}-X^{\ast}) B } \\
  & \geq \delta_k \bF{ A }^2\bF{ B }^2 \cdot \frac{\sigma_{r}^{2}(A) \sigma_{r}^{2}(B)}{\bF{ A }^2\bF{ B }^2}  \cdot
 \bF{ X^{(k)}-X^{\ast} }^2\\
  & \geq \widetilde{\rho}_k \bF{ X^{(k)}-X^{\ast} }^2,
 \end{align*}
 where the last line is from Proposition \ref{PmRGRK:gammak}. By taking full expectation on both sides of this inequality, the result in \eqref{PmRGRK+eq:E_Wij(X)} is obtained.
 \qed
\end{proof}

\section{Convergence analysis}\label{sec:PmRGRK+NmRGRK_convergence}

In this section, we will analyze the convergence of the PmRGRK and NmRGRK methods. First we present a lemma from \cite[Lemma 9]{20LR} which we will use in our convergence proofs.

\begin{lemma}\cite[Lemma 9]{20LR}\label{lemma:nonnegative+sequence}
Fix $F_1 = F_0 \geq 0$ and let $\{F_k\}_{k=0}^{\infty}$ be a sequence of nonnegative real numbers satisfying the relation
\begin{align*}
  F_{k+1} \leq  t_1 F_{k} + t_2 F_{k-1}
\end{align*}
for any $k\geq 1$, where $t_2 \geq 0$, $t_1 + t_2<1$, and at least one of the coefficients $t_1$ and $t_2$ is positive.
Then the sequence satisfies the relation
\begin{align*}
  F_{k+1} \leq  q_1^{k}(1 + q_2)F_0
\end{align*}
for all $k\geq 1$, where $q_1 = (t_1+\sqrt{t_1^2+4t_2})/2$ and $q_2 = q_1 -t_1\geq 0$.
\end{lemma}

\begin{theorem}\label{thm:CA-PmRGRK}
Let the matrix equation \eqref{AXB=C} be consistent. If the two initial guesses $X^{(0)}=X^{(1)}\in \mathbb{R}^{n\times n}$ with $vec(X^{(0)})$ being in the column space of $B\otimes A^T$, the step-size $0<\alpha<2$, and the momentum parameter $0<\beta<(\sqrt{\tau_1^2 + 12(1-\tau_2)}-\tau_1)/6$ with $\tau_1=4+\alpha-\alpha \widetilde{\rho}$ and $\tau_2=\alpha(2-\alpha)\widetilde{\rho}$, then the iteration sequence $\{X^{(k)}\}_{k=0}^{\infty}$, generated by Algorithm \ref{alg:PmRGRK}, satisfies
\begin{align}\label{thm:CA-PmRGRK+eq2}
  \Ec\left[ \BF{ X^{(k+1)} - X^{\ast} } \right]\leq
  \left( \frac{\sqrt{\gamma_1^2+4\gamma_2} + \gamma_1}{2}  \right)^k
  \left( \frac{\sqrt{\gamma_1^2+4\gamma_2} - \gamma_1}{2}  \right)
  \BF{ X^{(0)} - X^{\ast} },
\end{align}
where $\gamma_{1} = (1+3\beta+\beta^2) - (2\alpha + \alpha\beta - \alpha^2)\widetilde{\rho}$  and
$\gamma_{2} = 2\beta^2+(1+\alpha)\beta$.
\end{theorem}

\begin{proof}
For the sake of simplicity, we define the $k$th error matrix as $H^{(k)}=X^{(k)}-X^{\ast}$ for $k=0,1,2,\cdots$. We first divide   $\BF{H^{(k+1)}}$ into three parts, i.e.,
\begin{align}\label{eq1:PmRGRK_x_k+1}
  \BF{ H^{(k+1)} } = s_{k,1}  + s_{k,2} + s_{k,3},
\end{align}
where $s_{k,1}$,  $s_{k,2}$,  and  $s_{k,3}$ are respectively defined by
\begin{align*}
\left \{
\begin{array}{l}
   s_{k,1}  =  \BF{ H^{(k)} + \alpha  V_{i_k,j_k}(X^{(k)})  a_{i_k} b_{j_k}^T  },  \vspace{1ex}\\
   s_{k,2}  =  2\beta\left\langle H^{(k)} + \alpha  V_{i_k,j_k}(X^{(k)})  a_{i_k} b_{j_k}^T,~~ X^{(k)} - X^{(k-1)} \right\rangle_F,  \vspace{1ex} \\
   s_{k,3} = \beta^2 \BF{ X^{(k)} - X^{(k-1)} },
\end{array}
\right.
\end{align*}
with
\begin{align*}
  V_{i_k,j_k}(X^{(k)}) = \frac{C_{i_k,j_k} - a_{i_k}^T X^{(k)} b_{j_k}}{\norm{a_{i_k}}^2  \norm{b_{j_k}}^2}.
\end{align*}
We proceed to analyze them individually.

According to the fact that
\begin{align*}
 \BF{V_{i,j}(X^{(k)})  a_{i} b_{j}^T }
 = -\left\langle H^{(k)},~~ V_{i,j}(X^{(k)})  a_{i} b_{j}^T \right\rangle_F
 = W_{i,j}(X^{(k)}),
\end{align*}
we have
\begin{align}\label{eq:PmRGRK_sk1}
  s_{k,1}
  & = \BF{H^{(k)}}
      + 2\alpha \left\langle H^{(k)},~~ V_{i_k,j_k}(X^{(k)})  a_{i_k} b_{j_k}^T \right\rangle_F
      + \alpha^2 \BF{V_{i_k,j_k}(X^{(k)})  a_{i_k} b_{j_k}^T }\notag\\
  & = \BF{ H^{(k)} } + (\alpha^2 - 2\alpha) W_{i_k,j_k}(X^{(k)}).
\end{align}
Define two auxiliary variables
\begin{align*}
  s_{k,2}^{(1)}
  & = 2\beta\left\langle H^{(k)}, H^{(k)} \right\rangle_F  +
      2\beta\left\langle H^{(k)}, - H^{(k-1)} \right\rangle_F \\
  &\leq 2\beta\BF{H^{(k)}} +  \beta (\BF{H^{(k)}} + \BF{H^{(k-1)}})\\
  & = 3\beta\BF{H^{(k)}} + \beta \BF{ H^{(k-1)} }
\end{align*}
and
\begin{align*}
  s_{k,2}^{(2)}
  & = 2\alpha\beta\left\langle V_{i,j}(X^{(k)})  a_{i} b_{j}^T, H^{(k)} \right\rangle_F +
      2\alpha\beta\left\langle V_{i,j}(X^{(k)})  a_{i} b_{j}^T,  - H^{(k-1)} \right\rangle_F\\
  & =-2\alpha\beta W_{i_k,j_k}(X^{(k)})   +
      2\alpha\beta\left\langle V_{i,j}(X^{(k)})  a_{i} b_{j}^T, -H^{(k-1)} \right\rangle_F\\
  &\leq-2\alpha\beta W_{i_k,j_k}(X^{(k)})   +
       \alpha\beta (W_{i_k,j_k}(X^{(k)}) + \BF{H^{(k-1)}})\\
  & =  \alpha\beta \BF{H^{(k-1)}} -  \alpha\beta W_{i_k,j_k}(X^{(k)}).
\end{align*}
It follows that
\begin{align}\label{eq:PmRGRK_sk2}
  s_{k,2}
& = 2\beta\left\langle H^{(k)},  X^{(k)} - X^{(k-1)} \right\rangle_F +
    2\alpha\beta\left\langle V_{i,j}(X^{(k)})  a_{i} b_{j}^T, X^{(k)} - X^{(k-1)} \right\rangle_F\notag\\
& = 2\beta\left\langle H^{(k)}, H^{(k)} - H^{(k-1)} \right\rangle_F +
    2\alpha\beta\left\langle V_{i,j}(X^{(k)})  a_{i} b_{j}^T,  H^{(k)} - H^{(k-1)} \right\rangle_F\notag\\
&=s_{k,2}^{(1)}  +   s_{k,2}^{(2)}\notag\\
& \leq 3\beta\BF{ H^{(k)} } + (1+\alpha)\beta \BF{ H^{(k-1)} } -  \alpha\beta W_{i_k,j_k}(X^{(k)}).
\end{align}
Using the inequality $\BF{X-Y}\leq 2(\BF{X-Z} + \BF{Y-Z})$ for any matrices $X$, $Y$, and $Z$ with compatible dimension, it holds that
\begin{align}\label{eq:PmRGRK_sk3}
 s_{k,3} \leq 2\beta^2 \BF{ H^{(k)} } + 2\beta^2 \BF{ H^{(k-1)} }.
\end{align}

Combining formulas \eqref{eq:PmRGRK_sk1}, \eqref{eq:PmRGRK_sk2}, and \eqref{eq:PmRGRK_sk3}, it indicates that
\begin{align*}
  \BF{ H^{(k+1)}}
  &\leq (1+3\beta+\beta^2)\BF{ H^{(k)} }
  + (2\alpha + \alpha\beta - \alpha^2) (-W_{i_k,j_k}(X^{(k)}))\\
  &\quad + (2\beta^2+(1+\alpha)\beta)\BF{ H^{(k-1)} }.
\end{align*}
By first taking expectation with respect to $(i_k,j_k)\in\Delta_k$, we obtain
\begin{align*}
  \Ec_k\left[\BF{ H^{(k+1)}} \right]
  &\leq (1+3\beta+\beta^2)\Ec_k\left[\BF{ H^{(k)} } \right]
  + (2\alpha + \alpha\beta - \alpha^2) \Ec_k\left[-W_{i_k,j_k}(X^{(k)})\right]\\
  &\quad + (2\beta^2+(1+\alpha)\beta)\Ec_k\left[\BF{ H^{(k-1)} }\right]\\
  &\leq (1+3\beta+\beta^2)\Ec_k\left[\BF{ H^{(k)} } \right]
  - (2\alpha + \alpha\beta - \alpha^2) \widetilde{\rho}_k \bF{ H^{(k)} }^2\\
  &\quad + (2\beta^2+(1+\alpha)\beta)\Ec_k\left[\BF{ H^{(k-1)} }\right],
\end{align*}
where the second inequality is from formula \eqref{PmRGRK+eq:E_Wij(X)}.
By taking expectation again, we get the three-term recurrence relation
\begin{align*}
  \Ec\left[ \BF{ X^{(k+1)} - X^{\ast} } \right]\leq \gamma_{1} \Ec\left[ \BF{ X^{(k)} - X^{\ast} } \right] + \gamma_{2} \Ec\left[ \BF{ X^{(k-1)} - X^{\ast} } \right].
\end{align*}
Since $0<\alpha<2$ and $0<\beta<(\sqrt{\tau_1^2 + 12(1-\tau_2)}-\tau_1)/6$, we have $\gamma_{2}>0$ and
\begin{align*}
\gamma_{1} + \gamma_{2} = 3 \beta^2 + \tau_1 \beta + \tau_2 <1.
\end{align*}
Then, we obtain the convergence result in Theorem \ref{thm:CA-PmRGRK} from Lemma \ref{lemma:nonnegative+sequence}.
\qed
\end{proof}

\begin{remark}
  We note that the convergence factor in Theorem \ref{thm:CA-PmRGRK} is less than one. The inequality follows directly from the condition $\gamma_{1}+\gamma_{2}<1$. This is because,
 \begin{align*}
   \frac{\sqrt{\gamma_1^2+4\gamma_2} + \gamma_1}{2} -1
    = \frac{(\gamma_1 -2) + \sqrt{(\gamma_1-2)^2+4(\gamma_1+\gamma_2)-4}}{2}<0.
 \end{align*}
\end{remark}

Now, we turn to analyze the convergence of the NmRGRK method.

\begin{theorem}\label{thm:CA-NmRGRK}
Let the matrix equation \eqref{AXB=C} be consistent. If the two initial guesses $X^{(0)}=Y^{(0)}\in \mathbb{R}^{n\times n}$ with $vec(X^{(0)})$ being in the column space of $B\otimes A^T$, the step-size $0<\alpha<2$, and the momentum parameter $0<\beta<(\sqrt{2\tau_3 -1}-1)/2$ with $\tau_3 = 1/(2 + 2\alpha(2-\alpha)\widetilde{\rho})$, then the iteration sequence $\{X^{(k)}\}_{k=0}^{\infty}$, generated by Algorithm \ref{alg:NmRGRK}, satisfies
\begin{align}\label{thm:CA-NmRGRK+eq2}
  \Ec\left[ \BF{ X^{(k+1)} - X^{\ast} } \right]\leq
  \left( \frac{\sqrt{\gamma_3^2+4\gamma_4} + \gamma_3}{2}  \right)^k
  \left( \frac{\sqrt{\gamma_3^2+4\gamma_4} - \gamma_3}{2}  \right)
  \BF{ X^{(0)} - X^{\ast} },
\end{align}
where $\gamma_{3} = 2(1+\beta)^2 (1 + (\alpha^2 - 2\alpha) \widetilde{\rho})$  and
      $\gamma_{4} = 2\beta^2 (1 + (\alpha^2 - 2\alpha) \widetilde{\rho})$.
\end{theorem}

\begin{proof}
  The proof is almost the same as the proof of Theorem \ref{thm:CA-PmRGRK}, so we skip the repeated parts.

  We first write the $(k+1)$th NmRGRK squared error as follows.
  \begin{align}\label{eq:NmRGRK_sk4+5}
  \BF{ H^{(k+1)} } \leq 2 (s_{k,4}  + s_{k,5}),
\end{align}
where $s_{k,4}$ and  $s_{k,5}$ are respectively defined by
\begin{align*}
\left \{
\begin{array}{l}
   s_{k,4}  =  (1+\beta)^2 \BF{ Y^{(k+1)} - X^{\ast} },  \vspace{1ex}\\
   s_{k,5}  = \beta^2 \BF{ Y^{(k)} - X^{\ast} }.
\end{array}
\right.
\end{align*}
These two terms will be analyzed individually.

For the first term, we have
\begin{align*}
s_{k,4}
& =  (1+\beta)^2 \BF{ Y^{(k+1)} - X^{\ast} }\\
& =  (1+\beta)^2 \BF{ H^{(k)} + \alpha  V_{i_k,j_k}(X^{(k)})  a_{i_k} b_{j_k}^T  } \\
& =  (1+\beta)^2 \BF{ H^{(k)} } + (\alpha^2 - 2\alpha)(1+\beta)^2 W_{i_k,j_k}(X^{(k)}).
\end{align*}
Based on Proposition \ref{PmRGRK+eq:E_Wij(X)}, it yields that for any index pair $(i_k,j_k)\in\Delta_k$,
\begin{align*}
 \Ec \left[ W_{i,j}(X^{(k)}) \right] \geq \gamma_k \widetilde{\rho} \Ec \left[ \bF{ H^{(k)} }^2 \right] \geq  \widetilde{\rho} \Ec \left[ \bF{ H^{(k)} }^2 \right]
\end{align*}
By taking expectation for $s_{k,4}$, we have
\begin{align}\label{eq:NmRGRK_sk4}
\Ec \left[ s_{k,4} \right] \leq
(1+\beta)^2 (1 + (\alpha^2 - 2\alpha) \widetilde{\rho}) \Ec \left[ \bF{ H^{(k)} }^2 \right].
\end{align}
It is a similar story to show that for any $(i_{k-1},j_{k-1})\in\Delta_{k-1}$,
\begin{align*}
s_{k,5}
 =  \beta^2 \BF{ H^{(k-1)} } + (\alpha^2 - 2\alpha)\beta^2 W_{i_{k-1},j_{k-1}}(X^{(k-1)})
\end{align*}
and
\begin{align}\label{eq:NmRGRK_sk5}
\Ec \left[ s_{k,5} \right] \leq
\beta^2 (1 + (\alpha^2 - 2\alpha) \widetilde{\rho}) \Ec \left[ \bF{ H^{(k-1)} }^2 \right]
\end{align}
for $k=1,2,3,\cdots$.

Combining formulas \eqref{eq:NmRGRK_sk4+5}, \eqref{eq:NmRGRK_sk4}, and \eqref{eq:NmRGRK_sk5}, it leads to
\begin{align*}
\Ec \left[ \BF{ H^{(k+1)} } \right] \leq
\gamma_{3} \Ec \left[ \BF{ H^{(k)} } \right] + \gamma_{4} \Ec \left[ \BF{ H^{(k-1)} } \right].
\end{align*}
Since $0<\alpha<2$ and $0<\beta<(\sqrt{2\tau_3 -1}-1)/2$, we have $\gamma_{4}>0$ and
\begin{align*}
\gamma_{3} + \gamma_{4} - 1 = \frac{1}{\tau_3}(2 \beta^2 + 2 \beta + 1 - \tau_3) < 0.
\end{align*}
Then, the convergence result in Theorem \ref{thm:CA-NmRGRK} is obtained by applying Lemma \ref{lemma:nonnegative+sequence}.
\qed
\end{proof}

\section{Experimental results}\label{sec:PmRGRK+NmRGRK_numer}

In this section, we implement the ME-RGRK method \cite{22WLZ} and its two momentum variants, including PmRGRK and NmRGRK, and show that PmRGRK and NmRGRK are numerically advantageous over ME-RGRK in terms of the relative residual norm (RRN), the number of iteration steps (IT), and the computing time in seconds (CPU), where RRN is defined by
 \begin{align*}
  {\rm RRN} =\frac{\bF{ R^{(k)}} }{ \bF{ R^{(0)}}} = \frac{\bF{     C-AX^{(k)}B   } }{ \bF{ C-AX^{(0)}B   }}
 \end{align*}
and CPU is realized by applying MATLAB built-in function, e.g., {\sf tic-toc}. Note that the CPU and IT are the arithmetical averages of the elapsed CPU times and the required iteration steps concerning $20$ times repeated runs of the corresponding method, respectively, because of the randomness of the methods, where IT is taken as an integer. We also report the speed-up (SU) of PmRGRK or NmRGRK against ME-RGRK, which is defined by
\begin{align*}
 {\rm SU} = \frac{{\rm CPU~of~ME-RGRK}}{{\rm CPU~of~PmRGRK~or~NmRGRK}}.
\end{align*}
The relaxation parameter in these three methods is selected  by an exhaustive strategy, e.g., $\theta=0.5$, $0.7$, and $0.9$. It is an experimental finding that the parameter pairs $(\alpha,\beta)=(0.9,0.3)$ and $(0.8,0.5)$ are appropriate for the PmRGRK and NmRGRK methods, respectively, which can bring about a satisfactory convergence rate. Not especially specified, we will adopt this parameter selection approach for PmRGRK and NmRGRK in the following numerical test. All numerical tests are performed on a Founder desktop PC with Intel(R) Core(TM) i5-7500 CPU 3.40 GHz.

\subsection{Synthetic data} The following coefficient matrix is generated from synthetic data.

\begin{example}\label{Example_1}
For the given $m$, $n$, and $p$, we consider two dense coefficient matrices, which is randomly generated by the MALTAB built-in function, e.g., $A={\sf randn}(m,n)$ and $B={\sf randn}(n,p)$.
\end{example}

\begin{example}\label{Example_2}
Consider the coefficient matrices $A={\sf sprandn}(m,n,1/n)$ and $B={\sf randn}(n,p)$, where {\sf sprandn} is a MALTAB built-in function and creates a random $m$-by-$n$ sparse matrix with approximately $m$ normally distributed nonzero entries.
\end{example}

\begin{example}\label{Example_3}
In this example, we consider the coefficient matrices with block structure, generated by $A=[A_1~A_1 ; A_1~A_1]$ with $A_1 = {\sf rand}(m/2,n/2)$ and $B={\sf randn}(n,p)$.
\end{example}

\begin{example}\label{Example_4}
As in Du et al. \cite{20DSS}, for the given $x_1$, $x_2$, and $r_1$, we construct a dense matrix $R$ by $R=UDV^T$, where $U\in \mathbb{R}^{x_1\times r_1}$, $D\in \mathbb{R}^{r_1\times r_1}$, and $V\in \mathbb{R}^{x_2\times r_1}$. Using MATLAB colon notation, these matrices are generated by
  $[U,\sim]={\sf qr}({\sf randn}(x_1,r_1),0)$,
  $[V,\sim]= {\sf qr}({\sf randn}(x_2,r_1),0)$, and
  $D={\sf diag}(1 + 2*{\sf rand}(r_1,1))$.
The output matrix is abbreviated as
$ R =   {\sf{Smatrix}} (x_1, x_2, r)$.
 In this example, we initialize the coefficient matrices in \eqref{AXB=C} as
$A =  {\sf{Smatrix}} (m, n, r_1)$
 and
$ B =  {\sf{Smatrix}} (n, p, r_2)$.
\end{example}

In this subsection, we will solve the matrix equation \eqref{AXB=C} with the coefficient matrices $A$ and $B$ from Examples \ref{Example_1} -- \ref{Example_4}. One of the solution vectors $X^\ast \in \mathbb{R}^{n\times n}$ is generated by using the MATLAB function {\sf randn(n)}  and the right-hand side $C \in \mathbb{R}^{m\times p}$ is taken to be $AX^\ast B$. Our implementations are respectively started from
$X^{(0)} =\textsc{0}$, $X^{(0)}=X^{(1)}=\textsc{0}$, and $X^{(0)}=Y^{(0)}=\textsc{0}$ for the ME-RGRK, PmRGRK, and NmRGRK methods. All computations are terminated once  ${\rm RRN} \leq 1\times 10^{-5}$, or the number of iteration steps exceeds $1\times10^5$.

For the randomly generated coefficient matrices in Examples \ref{Example_1} -- \ref{Example_2}, which are of {\it full rank},
Tables \ref{tab:Table-1} -- \ref{tab:Table-2} record the detailed outcomes of the iteration counts and the computing times for ME-RGRK, PmRGRK, and NmRGRK methods. From these tables, we see that the PmRGRK and NmRGRK methods perform more efficiently than the ME-RGRK method in terms of both iteration step and CPU time with significant speed-up in all cases. Exactly, the SU is at least $1.77$ (resp., $1.81$)  and at most $2.14$ (resp., $2.69$) for the PmRGRK (resp., NmRGRK) method. It also can be concluded from the tables that for the PmRGRK and NmRGRK methods, the number of iteration steps is decreasing rapidly, while the CPU time is increasing gradually when $n$ and $p$ are fixed but $m$ is grown.

For the coefficient matrices in Examples \ref{Example_3} -- \ref{Example_4}, which are {\it rank-deficient}, we list the number of iteration steps and the computing time for the ME-RGRK, PmRGRK, and NmRGRK methods in Tables \ref{tab:Table-3} -- \ref{tab:Table-4-4}. The results in these two tables show numerical phenomena similar to the above. The PmRGRK and NmRGRK methods can always successfully compute an approximate solution to the matrix equation \eqref{AXB=C}. The tables reveal that the iteration count and computing time of the PmRGRK and NmRGRK methods are considerably smaller than those of the ME-RGRK method, with the largest speed-ups being $1.84$ (resp., $2.33$) in Example \ref{Example_3} and $1.91$ (resp., $2.40$) in Example \ref{Example_4} for the PmRGRK (resp., NmRGRK) method. Given the fixed $n$ and $p$, and the increasing $m$, the number of iteration steps for the PmRGRK and NmRGRK methods is decreasing speedily, while the computing time is growing steadily.

\begin{table}[!htb]
	\caption{IT and CPU for the ME-RGRK, PmRGRK, and NmRGRK methods with different $m$, $n$, and $p$ in Example \ref{Example_1}}
	\begin{center}
		\begin{tabular}{|c |c| c| c| c| c|}
			\hline
			\multicolumn{2}{|c|}{$(m,n,p)$} & (400,50,100) & (600,50,100) & (800,50,100)  & (1000,50,100)  \\ \hline
			\multicolumn{6}{|l|}{$\theta=0.5$} \\  \hline
			\multirow{2}*{ME-RGRK} &IT  &36151 & 30834 & 28469 & 23826 \\ \cline{2-6}
			\multirow{2}*{} &CPU  &19.53  & 21.43 & 24.95 & 28.16 \\ \cline{2-6}
            \hline
			\multirow{2}*{PmRGRK} &IT  &24674  &21548  & 19964 & 17260 \\ \cline{2-6}
			\multirow{2}*{} &CPU  &9.75  & 11.20 & 13.67 & 15.86\\ \cline{2-6}
            \hline
            \multicolumn{2}{|c|}{SU} &2.00  &1.91  & 1.83 & 1.77  \\ \hline
            \multirow{2}*{NmRGRK} &IT  &18733  &17166  & 16273 & 14228 \\ \cline{2-6}
			\multirow{2}*{} &CPU  &8.04  & 9.95 & 12.52 & 15.53\\ \cline{2-6}
            \hline
            \multicolumn{2}{|c|}{SU} &2.43  &2.15  & 2.00 & 1.81  \\ \hline
            \multicolumn{6}{|l|}{$\theta=0.7$} \\  \hline
			\multirow{2}*{ME-RGRK} &IT  &35369 & 26347 & 22960 & 22519 \\ \cline{2-6}
			\multirow{2}*{} &CPU  &18.70  & 17.82 & 19.18 & 21.61 \\ \cline{2-6}
            \hline
			\multirow{2}*{PmRGRK} &IT  &23423  &17806  & 15382 & 15430 \\ \cline{2-6}
			\multirow{2}*{} &CPU  &9.15  & 8.90 & 9.89 & 11.67\\ \cline{2-6}
            \hline
            \multicolumn{2}{|c|}{SU} &2.04  &2.00  & 1.94 & 1.85  \\ \hline
            \multirow{2}*{NmRGRK} &IT  &17178  &13448  & 11916 & 12393 \\ \cline{2-6}
			\multirow{2}*{} &CPU  &7.20  & 7.57 & 8.64 & 11.08\\ \cline{2-6}
            \hline
            \multicolumn{2}{|c|}{SU} &2.60  &2.35  & 2.22 & 1.95  \\ \hline
            \multicolumn{6}{|l|}{$\theta=0.9$} \\  \hline
			\multirow{2}*{ME-RGRK} &IT  &34230  & 25954 & 21947 & 21214 \\ \cline{2-6}
			\multirow{2}*{} &CPU  &16.86  & 16.21 & 17.21 & 19.02 \\ \cline{2-6}
            \hline
			\multirow{2}*{PmRGRK} &IT  &22970  &17517 & 14955 & 14584 \\ \cline{2-6}
			\multirow{2}*{} &CPU  &8.48 & 8.30 &9.26 & 10.53\\ \cline{2-6}
            \hline
            \multicolumn{2}{|c|}{SU} &1.99  &1.95  & 1.86 & 1.81  \\ \hline
            \multirow{2}*{NmRGRK} &IT  &16575  &13070  & 11162 & 11210 \\ \cline{2-6}
			\multirow{2}*{} &CPU  &6.59 & 6.85 & 7.80 & 9.63\\ \cline{2-6}
            \hline
            \multicolumn{2}{|c|}{SU} &2.56  &2.37  & 2.21 & 1.97  \\ \hline
		\end{tabular}
	\end{center}
	\label{tab:Table-1}
\end{table}

\begin{table}[!htb]
	\caption{IT and CPU for the ME-RGRK, PmRGRK, and NmRGRK methods with different $m$, $n$, and $p$ in Example \ref{Example_2}}
	\begin{center}
		\begin{tabular}{|c |c| c| c| c| c|}
			\hline
 			\multicolumn{2}{|c|}{$(m,n,p)$} & (400,50,100) & (600,50,100) & (800,50,100)  & (1000,50,100)  \\ \hline
			\multicolumn{6}{|l|}{$\theta=0.5$} \\  \hline
			\multirow{2}*{ME-RGRK} &IT  &46885  & 43455 & 34653 & 34119 \\ \cline{2-6}
			\multirow{2}*{} &CPU  &29.88  & 36.70 & 36.26 & 42.62 \\ \cline{2-6}
            \hline
			\multirow{2}*{PmRGRK} &IT  &30737  &27831  & 22456 & 22076 \\ \cline{2-6}
			\multirow{2}*{} &CPU  &13.96  & 16.95 & 18.16 & 21.62\\ \cline{2-6}
            \hline
            \multicolumn{2}{|c|}{SU} &2.14  &2.17  & 2.00 & 1.97  \\ \hline
            \multirow{2}*{NmRGRK} &IT  &23154  &20982  & 17662 & 17275 \\ \cline{2-6}
			\multirow{2}*{} &CPU  &11.12  & 13.92 & 15.47 & 19.17\\ \cline{2-6}
            \hline
            \multicolumn{2}{|c|}{SU} &2.69  &2.64  & 2.34 & 2.22  \\ \hline
            \multicolumn{6}{|l|}{$\theta=0.7$} \\  \hline
			\multirow{2}*{ME-RGRK} &IT  &41276 & 37349 & 33358 & 33344 \\ \cline{2-6}
			\multirow{2}*{} &CPU  &23.85  & 26.77 & 36.10 & 38.27 \\ \cline{2-6}
            \hline
			\multirow{2}*{PmRGRK} &IT  &26353  &24134 & 21894 & 21843 \\ \cline{2-6}
			\multirow{2}*{} &CPU  &11.76  & 13.81 & 18.54 & 19.95\\ \cline{2-6}
            \hline
            \multicolumn{2}{|c|}{SU} &2.03  &1.94  & 1.95 & 1.92  \\ \hline
            \multirow{2}*{NmRGRK} &IT  &19707  &18115  & 17120 & 16538 \\ \cline{2-6}
			\multirow{2}*{} &CPU  &9.09  & 11.32 & 15.79 & 17.35\\ \cline{2-6}
            \hline
            \multicolumn{2}{|c|}{SU} &2.63  &2.37  & 2.29 & 2.21  \\ \hline
            \multicolumn{6}{|l|}{$\theta=0.9$} \\  \hline
			\multirow{2}*{ME-RGRK} &IT  &40545 & 34306 & 31566 & 31043 \\ \cline{2-6}
			\multirow{2}*{} &CPU  &24.66  & 23.45 & 30.20 &34.25 \\ \cline{2-6}
            \hline
			\multirow{2}*{PmRGRK} &IT  &26541  &23079  & 21313 & 20725 \\ \cline{2-6}
			\multirow{2}*{} &CPU  &11.70  & 12.89 & 15.94 & 18.48 \\ \cline{2-6}
            \hline
            \multicolumn{2}{|c|}{SU} &2.11  &1.82  & 1.89 & 1.85  \\ \hline
            \multirow{2}*{NmRGRK} &IT  &19345  &17336  & 15907 & 15955 \\ \cline{2-6}
			\multirow{2}*{} &CPU  &9.33  & 10.82 & 13.34& 16.08\\ \cline{2-6}
            \hline
            \multicolumn{2}{|c|}{SU} &2.64  &2.17  & 2.26 & 2.13  \\ \hline
		\end{tabular}
	\end{center}
	\label{tab:Table-2}
\end{table}

\begin{table}[!htb]
	\caption{IT and CPU for the ME-RGRK, PmRGRK, and NmRGRK methods with different $m$, $n$, and $p$ in Example \ref{Example_3}}
	\begin{center}
		\begin{tabular}{|c |c| c| c| c| c|}
			\hline
            \multicolumn{2}{|c|}{$(m,n,p)$} & (400,50,100) & (600,50,100) & (800,50,100)  & (1000,50,100)  \\ \hline
			\multicolumn{6}{|l|}{$\theta=0.5$} \\  \hline
			\multirow{2}*{ME-RGRK} &IT  &17209  & 16749 & 15124 & 13757 \\ \cline{2-6}
			\multirow{2}*{} &CPU  &7.15 & 8.42 & 9.07& 9.48 \\ \cline{2-6}
            \hline
			\multirow{2}*{PmRGRK} &IT  &11981  &11535  & 10829 & 8765 \\ \cline{2-6}
			\multirow{2}*{} &CPU  &3.84  & 4.74 &5.38 & 5.59 \\ \cline{2-6}
            \hline
            \multicolumn{2}{|c|}{SU} &1.86  &1.78  & 1.69 & 1.69  \\ \hline
            \multirow{2}*{NmRGRK} &IT  &9401  &9165  &8781  & 8064 \\ \cline{2-6}
			\multirow{2}*{} &CPU  &3.27  & 4.08 & 4.82 & 5.57 \\ \cline{2-6}
            \hline
            \multicolumn{2}{|c|}{SU} &2.19  &2.06  & 1.88 & 1.70  \\ \hline
            \multicolumn{6}{|l|}{$\theta=0.7$} \\  \hline
			\multirow{2}*{ME-RGRK} &IT  &16732  & 14813 & 13283 & 12238 \\ \cline{2-6}
			\multirow{2}*{} &CPU  &6.00  & 6.68 & 7.80 & 8.16 \\ \cline{2-6}
            \hline
			\multirow{2}*{PmRGRK} &IT  &11251  &9910  & 8964 & 8368 \\ \cline{2-6}
			\multirow{2}*{} &CPU  &3.20  & 3.60 & 4.35 & 4.79 \\ \cline{2-6}
            \hline
            \multicolumn{2}{|c|}{SU} &1.87  &1.86  & 1.80 & 1.70 \\ \hline
            \multirow{2}*{NmRGRK} &IT  &8471  &7942  & 7088 & 6558 \\ \cline{2-6}
			\multirow{2}*{} &CPU  &2.60 & 3.08 & 3.84 &4.54 \\ \cline{2-6}
            \hline
            \multicolumn{2}{|c|}{SU} &2.31  &2.17 & 2.03 & 1.80  \\ \hline
            \multicolumn{6}{|l|}{$\theta=0.9$} \\  \hline
			\multirow{2}*{ME-RGRK} &IT  &16337  & 14047 & 12752 & 12092 \\ \cline{2-6}
			\multirow{2}*{} &CPU  &5.43  & 6.00 & 6.90 &6.56\\ \cline{2-6}
            \hline
			\multirow{2}*{PmRGRK} &IT  &11009  &9563  & 8643 & 8304 \\ \cline{2-6}
			\multirow{2}*{} &CPU  &2.95  & 3.29 & 3.86 & 3.70\\ \cline{2-6}
            \hline
            \multicolumn{2}{|c|}{SU} &1.84  &1.82  & 1.79 & 1.77  \\ \hline
            \multirow{2}*{NmRGRK} &IT  &8043 &7061 & 6662 & 6237 \\ \cline{2-6}
			\multirow{2}*{} &CPU  &2.33  & 2.70 & 3.30 & 3.12\\ \cline{2-6}
            \hline
            \multicolumn{2}{|c|}{SU} &2.33  &2.23 & 2.09 & 2.10  \\ \hline
		\end{tabular}
	\end{center}
	\label{tab:Table-3}
\end{table}

\begin{table}[!htb]
	\caption{IT and CPU for the ME-RGRK, PmRGRK, and NmRGRK methods  with different $m$, $n$, and $p$ in Example \ref{Example_4}, where $r_1=r_2=40$.}
	\begin{center}
		\begin{tabular}{|c |c| c| c| c| c|}
			\hline
            \multicolumn{2}{|c|}{$(m,n,p)$} & (400,50,100) & (600,50,100) & (800,50,100)  & (1000,50,100) \\ \hline
			\multicolumn{6}{|l|}{$\theta=0.5$} \\  \hline
			\multirow{2}*{ME-RGRK} &IT  &17827  & 16885 & 14880 & 13453 \\ \cline{2-6}
			\multirow{2}*{} &CPU  &6.49  & 8.12 & 8.99 & 9.70 \\ \cline{2-6}
            \hline
			\multirow{2}*{PmRGRK} &IT  &12434  &11695  & 10391 & 9593\\ \cline{2-6}
			\multirow{2}*{} &CPU  &3.58  & 4.43 & 5.14 &5.88 \\ \cline{2-6}
            \hline
            \multicolumn{2}{|c|}{SU} &1.81  &1.83  & 1.75 & 1.65  \\ \hline
            \multirow{2}*{NmRGRK} &IT  &9607  &8897  &8011 & 7661 \\ \cline{2-6}
			\multirow{2}*{} &CPU  &2.97  &3.71 &4.31 & 5.65 \\ \cline{2-6}
            \hline
            \multicolumn{2}{|c|}{SU} &2.19  &2.19  & 2.09 & 1.72  \\ \hline
            \multicolumn{6}{|l|}{$\theta=0.7$} \\  \hline
			\multirow{2}*{ME-RGRK} &IT  &16939  & 13783 & 13541& 13193\\ \cline{2-6}
			\multirow{2}*{} &CPU  &6.08  &6.41 & 8.07& 8.39 \\ \cline{2-6}
            \hline
			\multirow{2}*{PmRGRK} &IT  &11687  &9376  & 9421 & 8977\\ \cline{2-6}
			\multirow{2}*{} &CPU  &3.29  &3.44 & 4.61 & 5.32\\ \cline{2-6}
            \hline
            \multicolumn{2}{|c|}{SU} &1.85  &1.86  & 1.75 & 1.58  \\ \hline
            \multirow{2}*{NmRGRK} &IT  &8712  &7183  & 7323 & 6749 \\ \cline{2-6}
			\multirow{2}*{} &CPU  &2.64  & 2.87 & 4.04 & 4.80\\ \cline{2-6}
            \hline
            \multicolumn{2}{|c|}{SU} &2.30  &2.23  & 2.00 & 1.75  \\ \hline
            \multicolumn{6}{|l|}{$\theta=0.9$} \\  \hline
			\multirow{2}*{ME-RGRK} &IT  &15909  & 13454& 12107 & 12851 \\ \cline{2-6}
			\multirow{2}*{} &CPU  &5.43  & 5.78 & 6.70 & 8.41\\ \cline{2-6}
            \hline
			\multirow{2}*{PmRGRK} &IT  &10719 &9141  & 8280 & 8850  \\ \cline{2-6}
			\multirow{2}*{} &CPU  &2.85  & 3.08 & 3.74 & 4.82 \\ \cline{2-6}
            \hline
            \multicolumn{2}{|c|}{SU} &1.91  &1.88  & 1.79 & 1.74  \\ \hline
            \multirow{2}*{NmRGRK} &IT  &7728  &6647 & 6191 & 6424 \\ \cline{2-6}
			\multirow{2}*{} &CPU  &2.26  & 2.50 & 3.23 & 4.30\\ \cline{2-6}
            \hline
            \multicolumn{2}{|c|}{SU} & 2.40 &2.32  &2.08 & 1.95 \\ \hline
		\end{tabular}
	\end{center}
	\label{tab:Table-4-4}
\end{table}

The above observations are intuitively illustrated in Figures \ref{fig:PmRGRK+NmRGRK_Example-1+RRNvsIT+CPU} -- \ref{fig:PmRGRK+NmRGRK_Example-4+RRNvsIT+CPU},  which depict the curves of the relative residual norm versus the iteration step and the computing time for Examples \ref{Example_1} -- \ref{Example_4}, respectively,  with $m=1200$, $n=50$, and $p=100$. According to  this figure, as the iteration step and computing time increase, the relative residual norm for the PmRGRK and NmRGRK methods decreases more rapidly than the ME-RGRK method. In all convergence cases, the NmRGRK method has the fastest convergence rate and costs the least computing time. Even though the PmRGRK and NmRGRK methods use more floats than the ME-RGRK method at each iteration, the computing times for the PmRGRK and NmRGRK methods to fixed accuracy are significantly lower than the ME-RGRK method, mainly due to the momentum acceleration.

\begin{figure}[!htb]
\centering
    \subfigure{
		\includegraphics[width=0.48\textwidth]{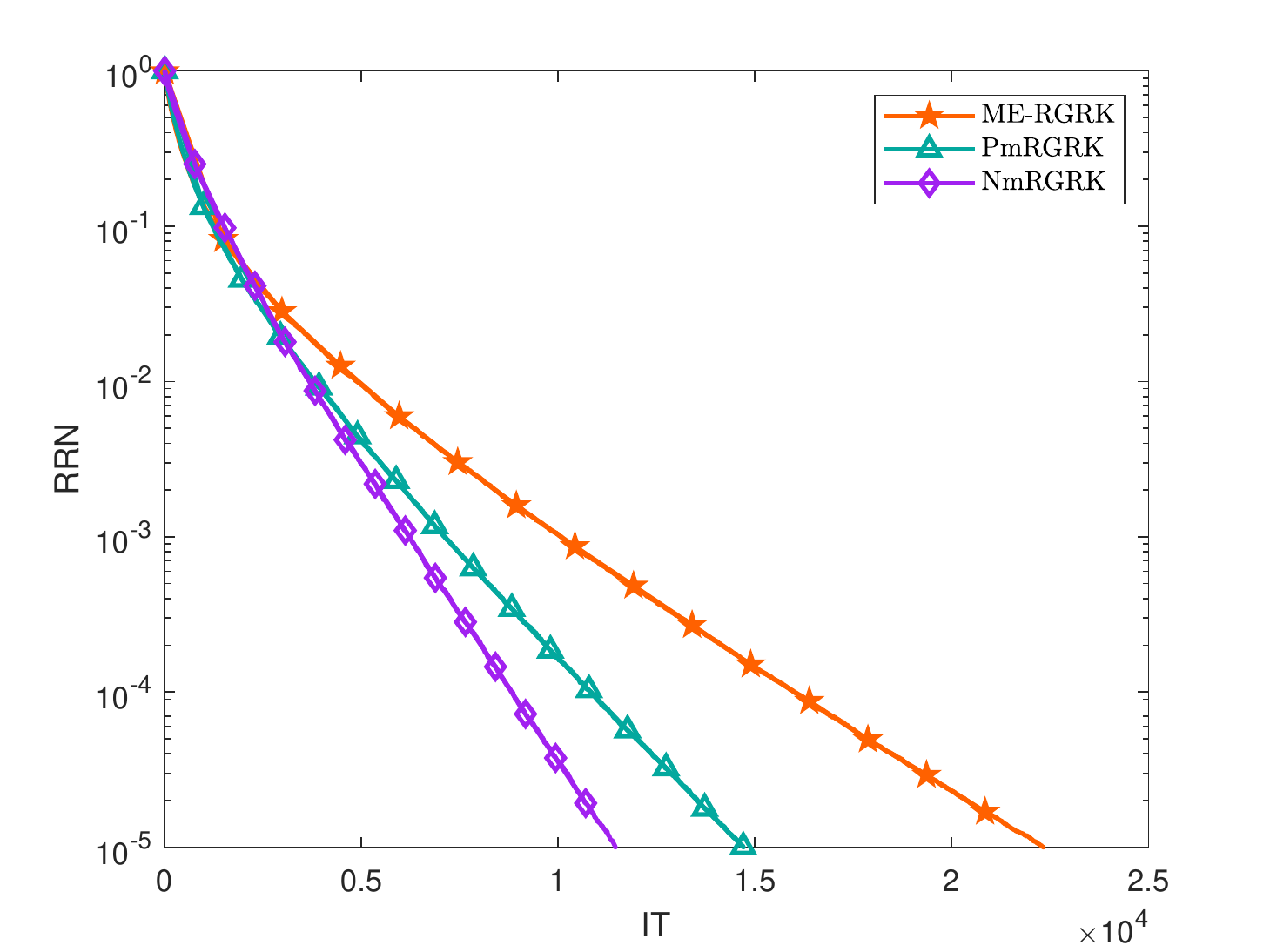}}
    \subfigure{
		\includegraphics[width=0.48\textwidth]{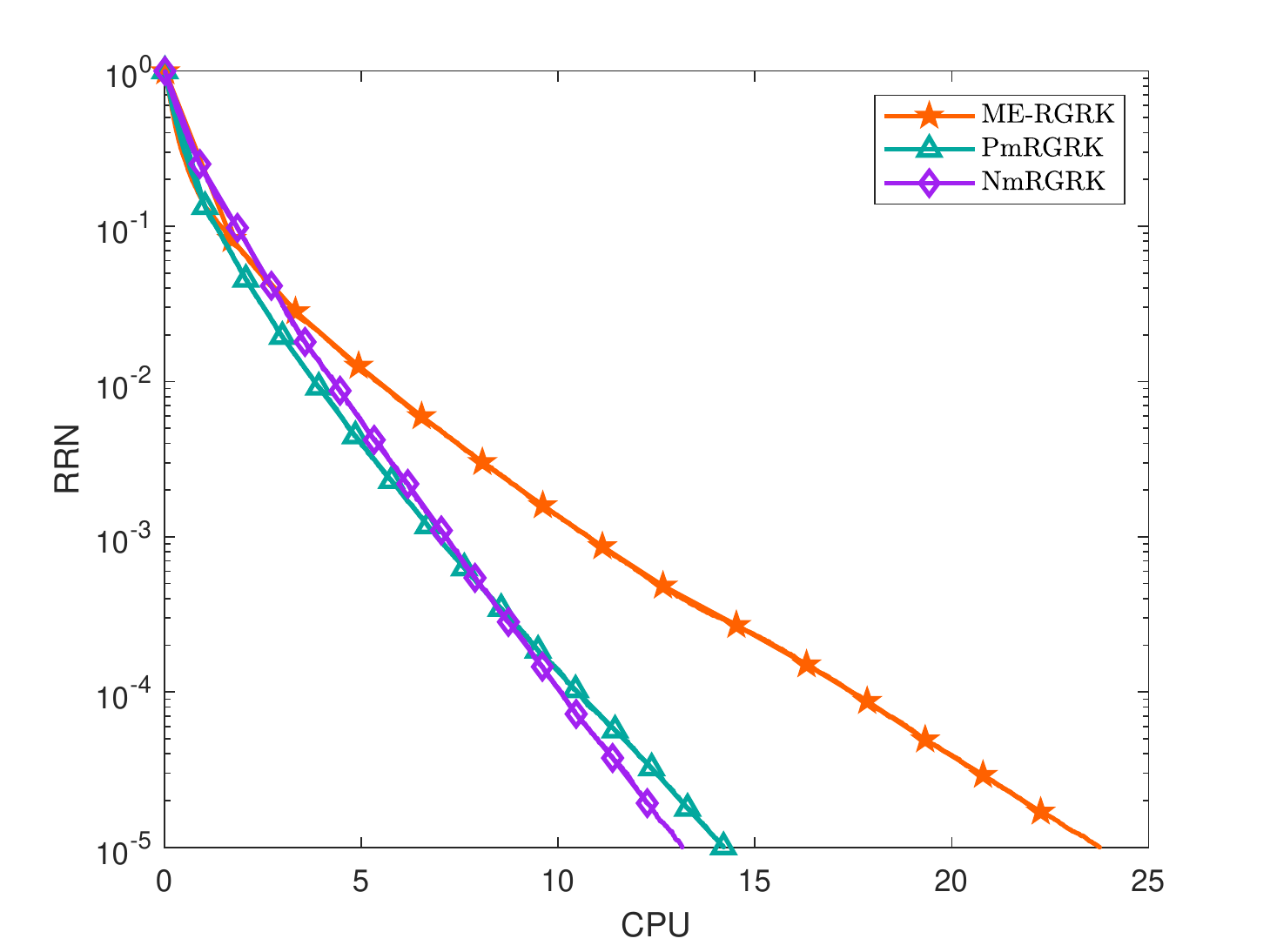}}
\caption{The convergence behaviors of RRN versus IT (left) and CPU (right) given by the ME-RGRK, PmRGRK, and NmRGRK methods for Example \ref{Example_1} with $m=1200$, $n=50$, $p=100$, and $\theta=0.9$.}
\label{fig:PmRGRK+NmRGRK_Example-1+RRNvsIT+CPU}
\end{figure}

\begin{figure}[!htb]
\centering
    \subfigure{
		\includegraphics[width=0.48\textwidth]{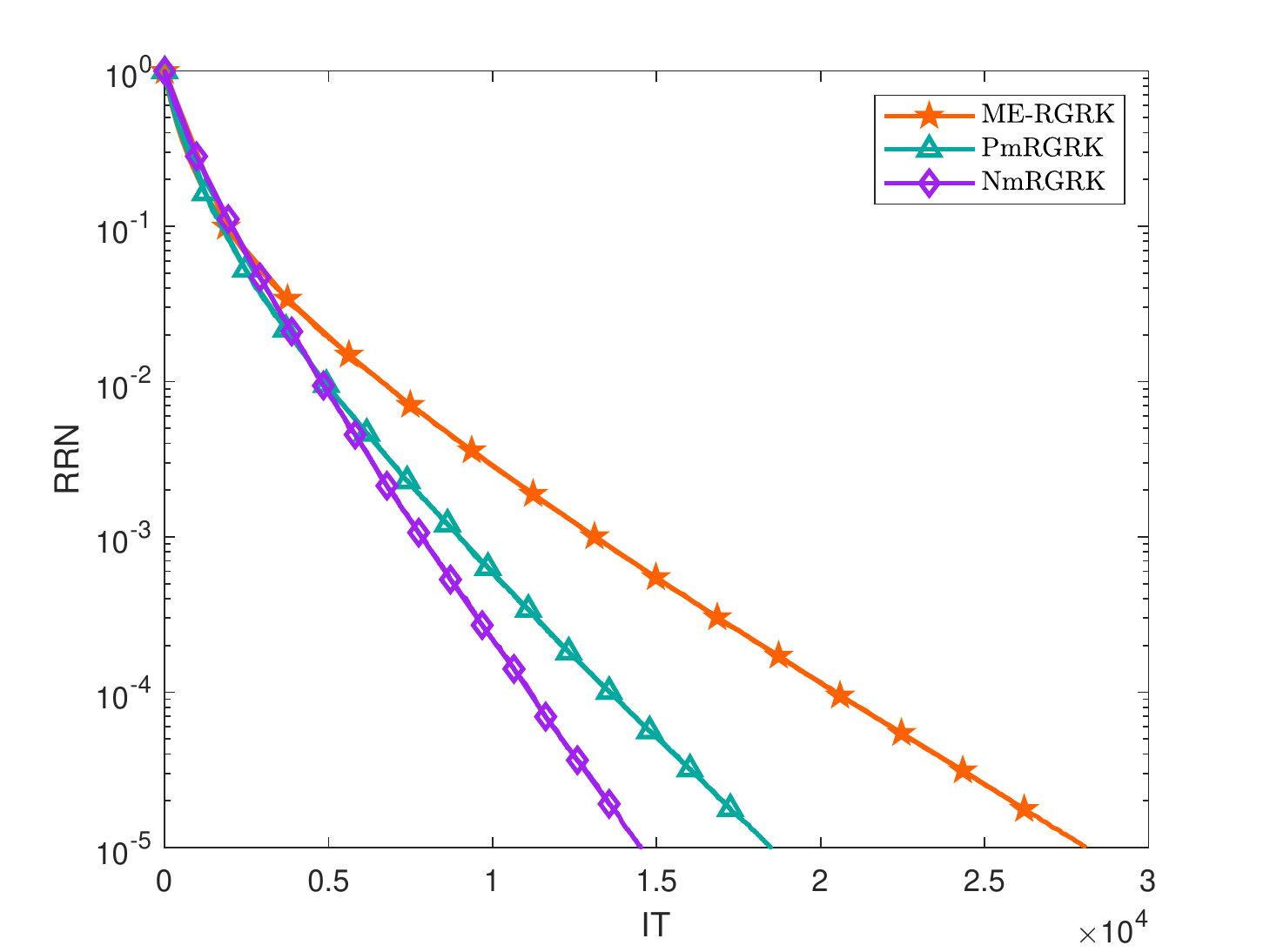}}
    \subfigure{
		\includegraphics[width=0.48\textwidth]{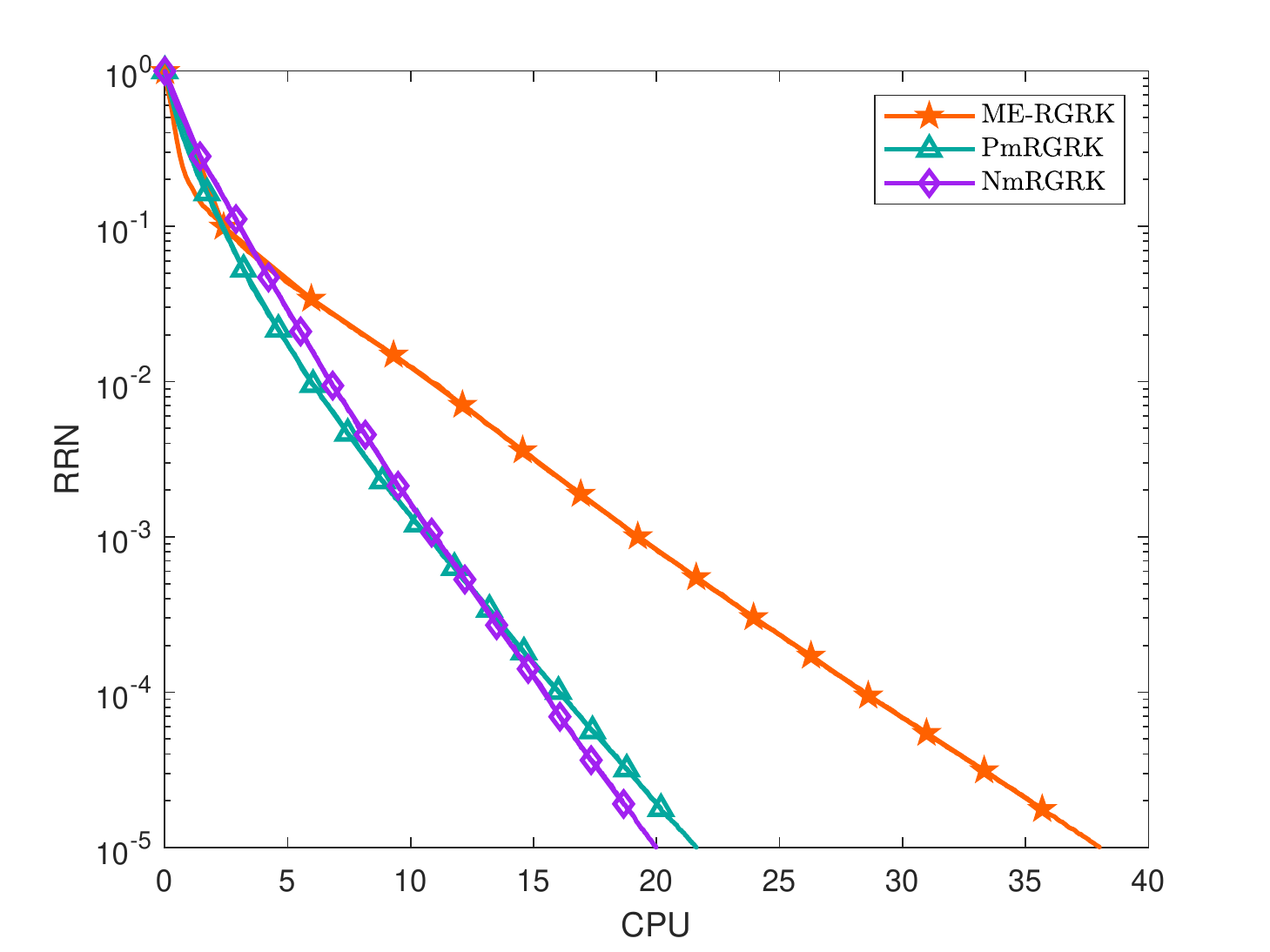}}
\caption{The convergence behaviors of RRN versus IT (left) and CPU (right) given by the ME-RGRK, PmRGRK, and NmRGRK methods for Example \ref{Example_2} with $m=1200$, $n=50$, $p=100$, and $\theta=0.9$.}
\label{fig:PmRGRK+NmRGRK_Example-2+RRNvsIT+CPU}
\end{figure}

\begin{figure}[!htb]
\centering
    \subfigure{
		\includegraphics[width=0.48\textwidth]{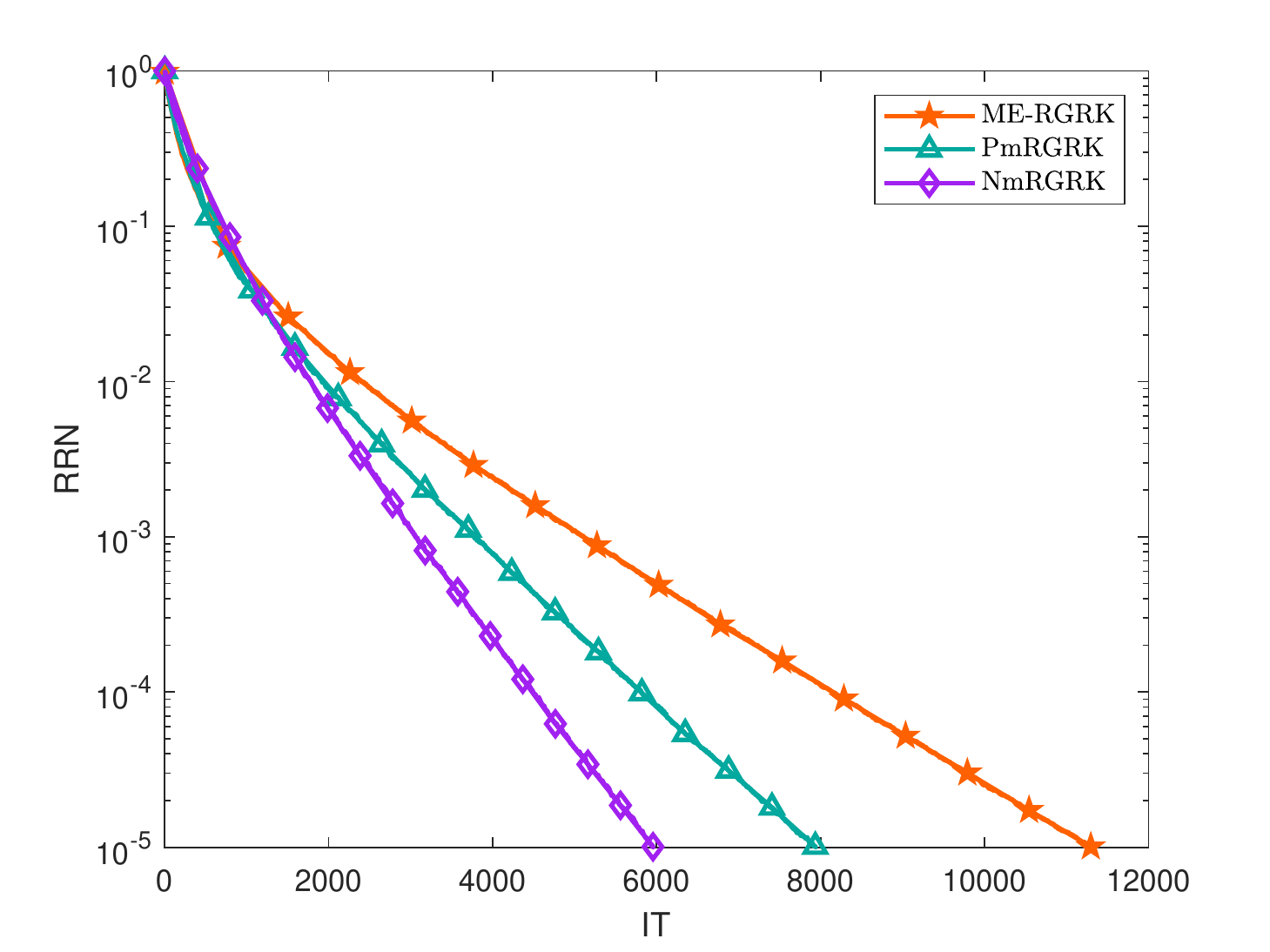}}
    \subfigure{
		\includegraphics[width=0.48\textwidth]{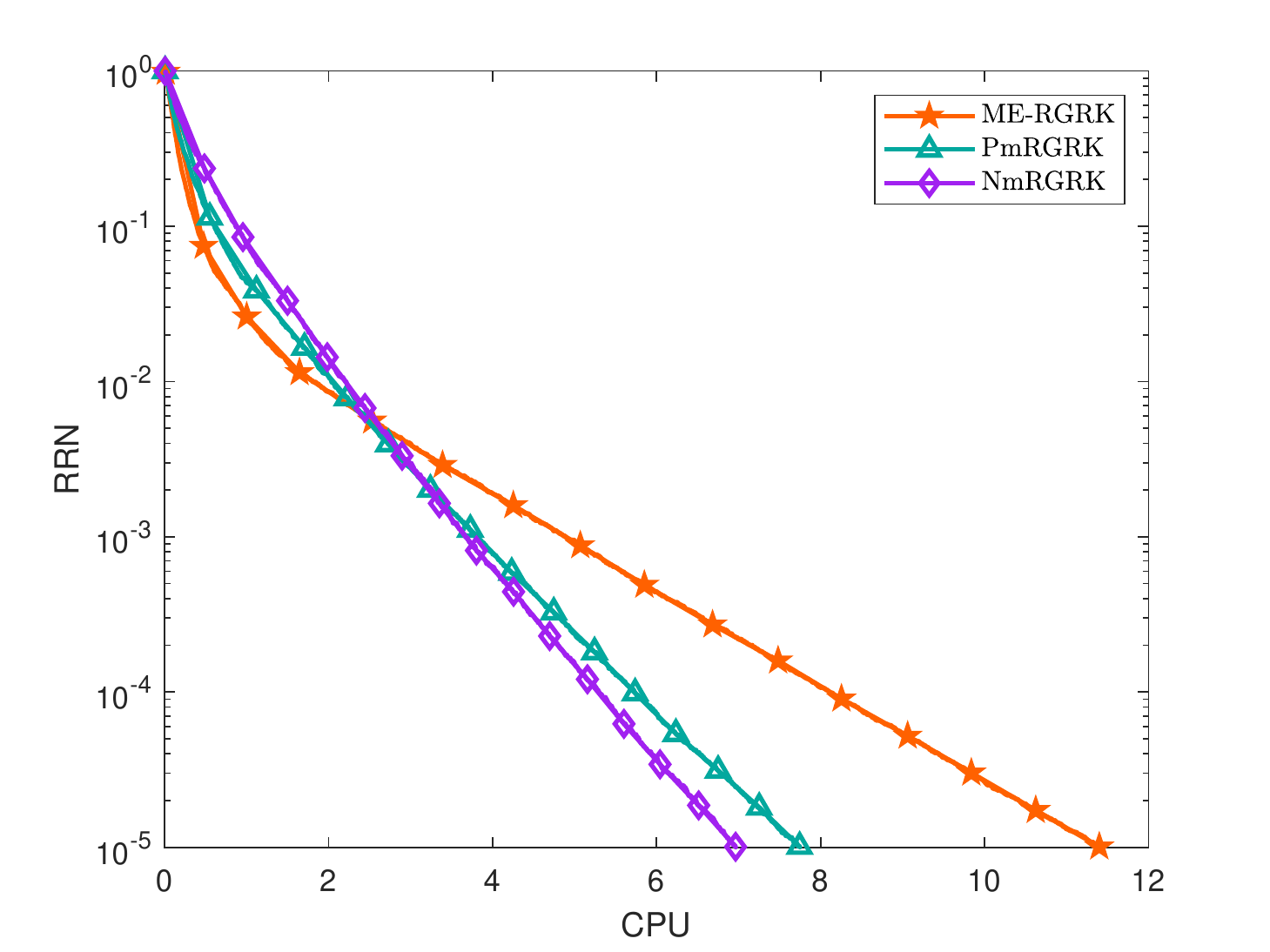}}
\caption{The convergence behaviors of RRN versus IT (left) and CPU (right) given by the ME-RGRK, PmRGRK, and NmRGRK methods for Example \ref{Example_3} with $m=1200$, $n=50$, $p=100$, and $\theta=0.9$.}
\label{fig:PmRGRK+NmRGRK_Example-3+RRNvsIT+CPU}
\end{figure}

\begin{figure}[!htb]
\centering
    \subfigure{
		\includegraphics[width=0.48\textwidth]{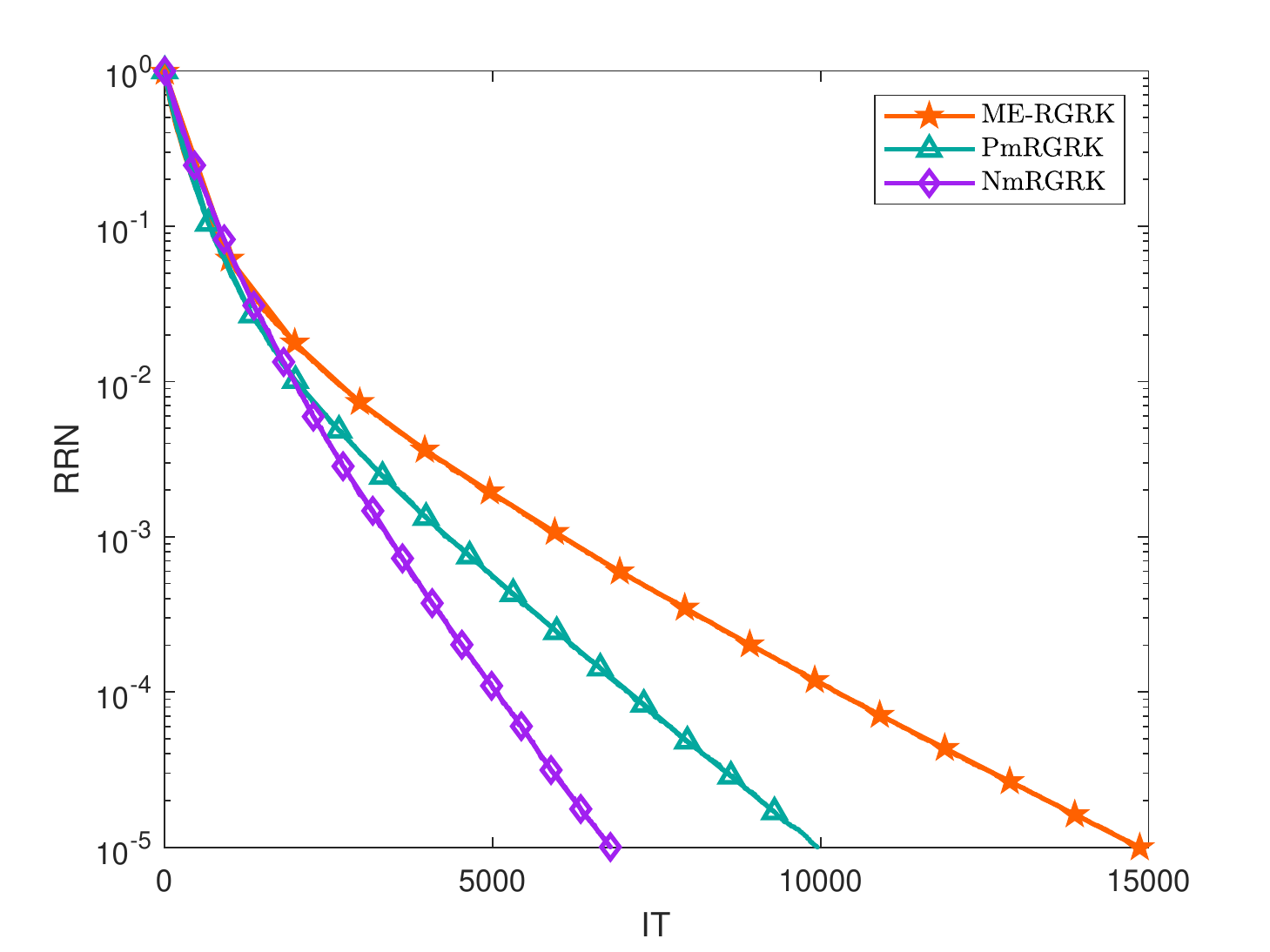}}
    \subfigure{
		\includegraphics[width=0.48\textwidth]{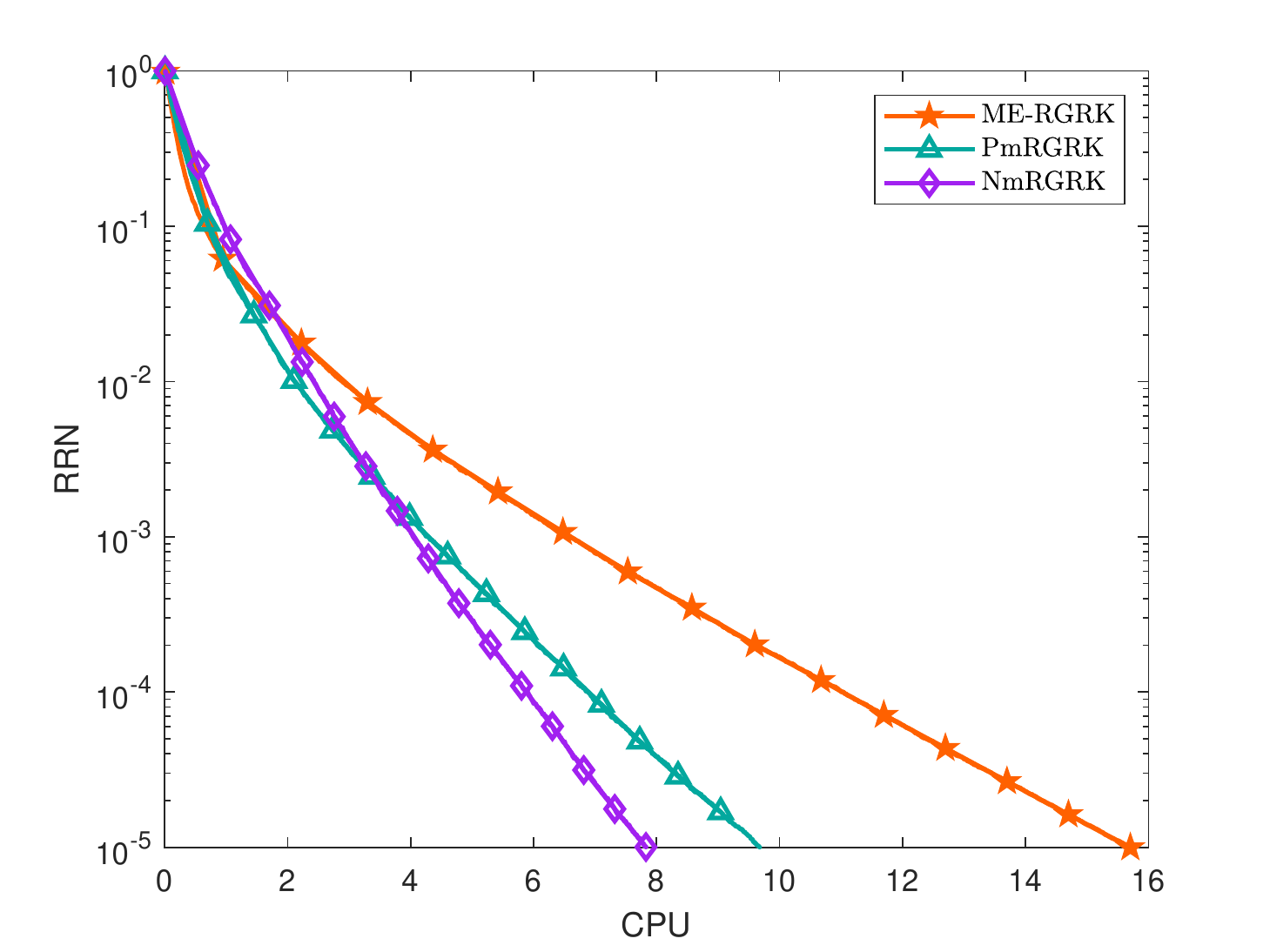}}
\caption{The convergence behaviors of RRN versus IT (left) and CPU (right) given by the ME-RGRK, PmRGRK, and NmRGRK methods for Example \ref{Example_4} with $m=1200$, $n=50$, $p=100$, and $\theta=0.9$.}
\label{fig:PmRGRK+NmRGRK_Example-4+RRNvsIT+CPU}
\end{figure}

We emphasize that the experimentally iterative parameter pair $(\alpha,\beta)$ for PmRGRK and NmRGRK, used in the above experiment, is not optimal. We may choose them to be any positive constant bounded by Theorems \ref{thm:CA-PmRGRK} -- \ref{thm:CA-NmRGRK}. Selecting an appropriate pair of iteration parameters may allow the momentum method to converge more quickly. For the sake of illustration, we utilize the PmRGRK and NmRGRK methods with various $(\alpha,\beta)$ to resolve the matrix equation in Example \ref{Example_4} and depict their convergence behaviors of RRN versus IT in Figures \ref{fig:PmRGRK_Example_4-alpha+beta} -- \ref{fig:NmRGRK_Example_4-alpha+beta}, respectively. In this example, we assign values to the input parameters as $m=400$, $n=50$, $p = 100$, and $r_1=r_2=40$, and to $(\alpha,\beta)$ as $(0.5,0.6)$, $(0.7,0.6)$, $(1.0,0.4)$, $(1.1,0.3)$, $(1.3,0.1)$, and $(1.3,0.3)$. It can be seen that both the PmRGRK and NmRGRK methods successfully compute an approximate solution for all cases. A satisfactory parameter pair is chosen by $(\alpha,\beta) = (1.0,0.4)$ and $(1.1,0.3)$ for the PmRGRK and NmRGRK methods, respectively. In this case, the iteration counts of the PmRGRK (resp., NmRGRK) method are $10078$, $8249$, and $6863$ (resp., $7671$, $7196$, and $6335$) when $\theta=0.5$, $0.7$, and $0.9$, respectively, which are appreciably smaller than those in Table \ref{tab:Table-4-4}.

\begin{figure}[!htb]
    \subfigure[$\theta=0.5$]{
		\includegraphics[width=0.48\textwidth]{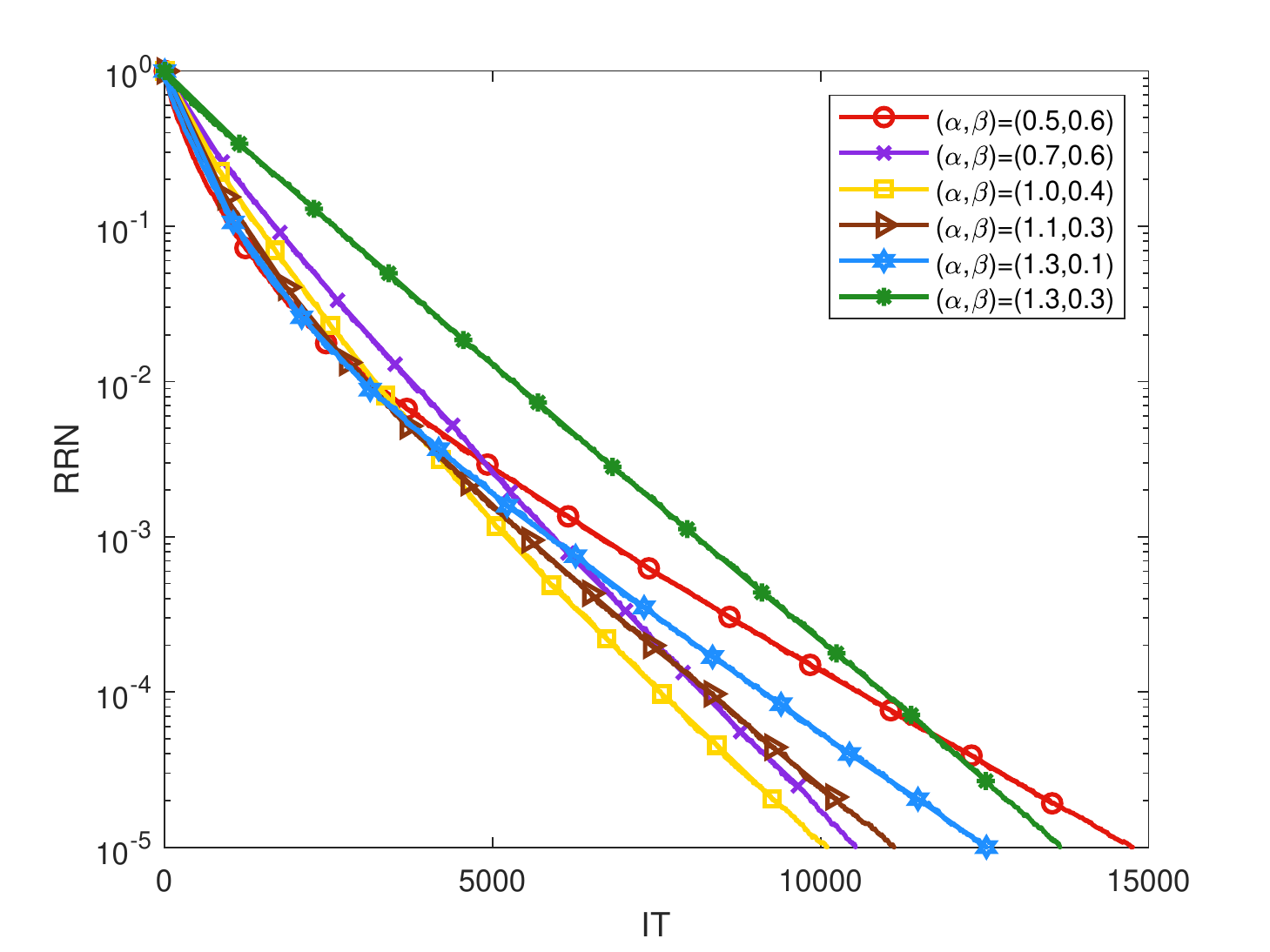}}
    \subfigure[$\theta=0.7$]{
		\includegraphics[width=0.48\textwidth]{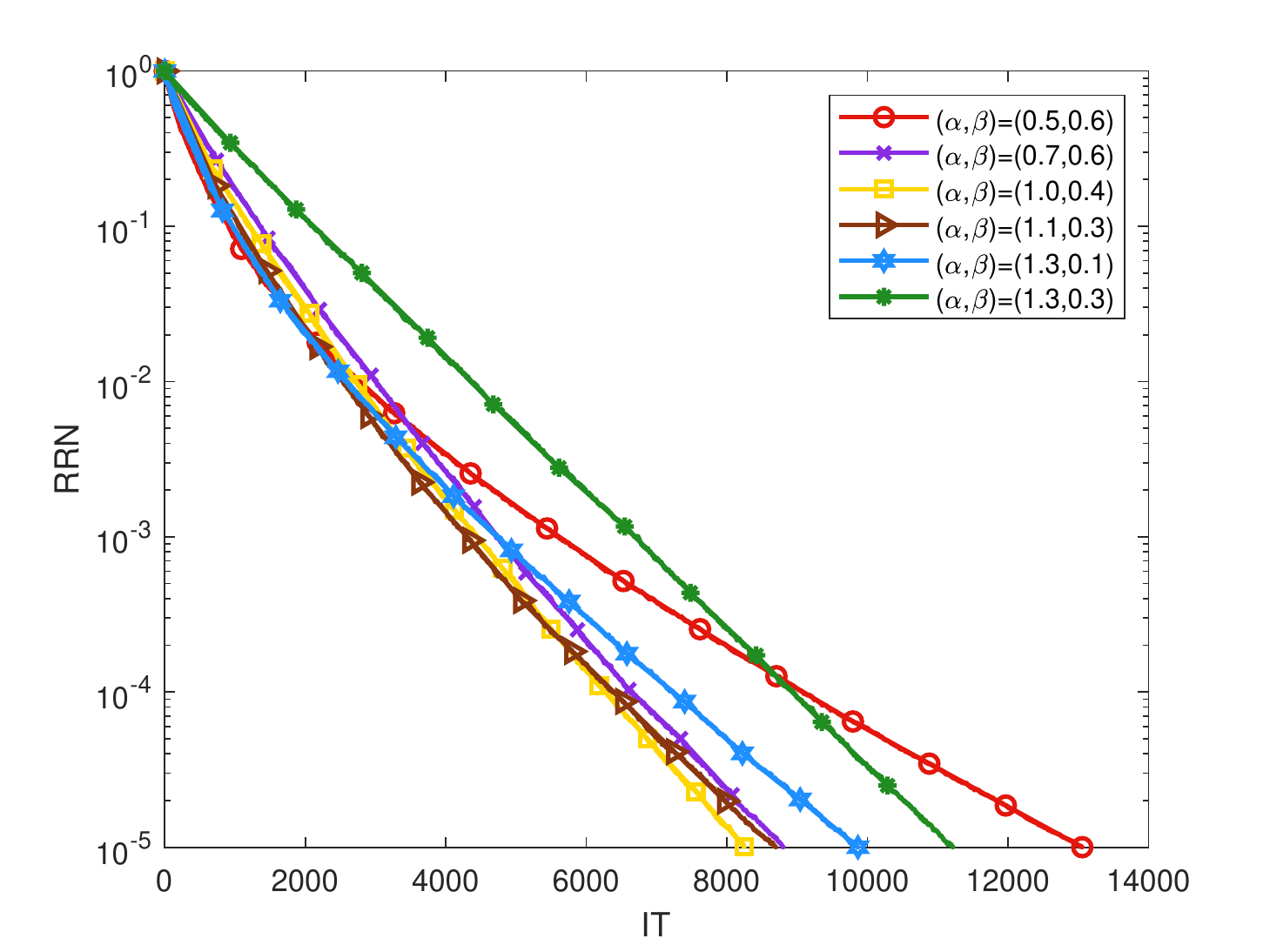}}
    \subfigure[$\theta=0.9$]{
		\includegraphics[width=0.48\textwidth]{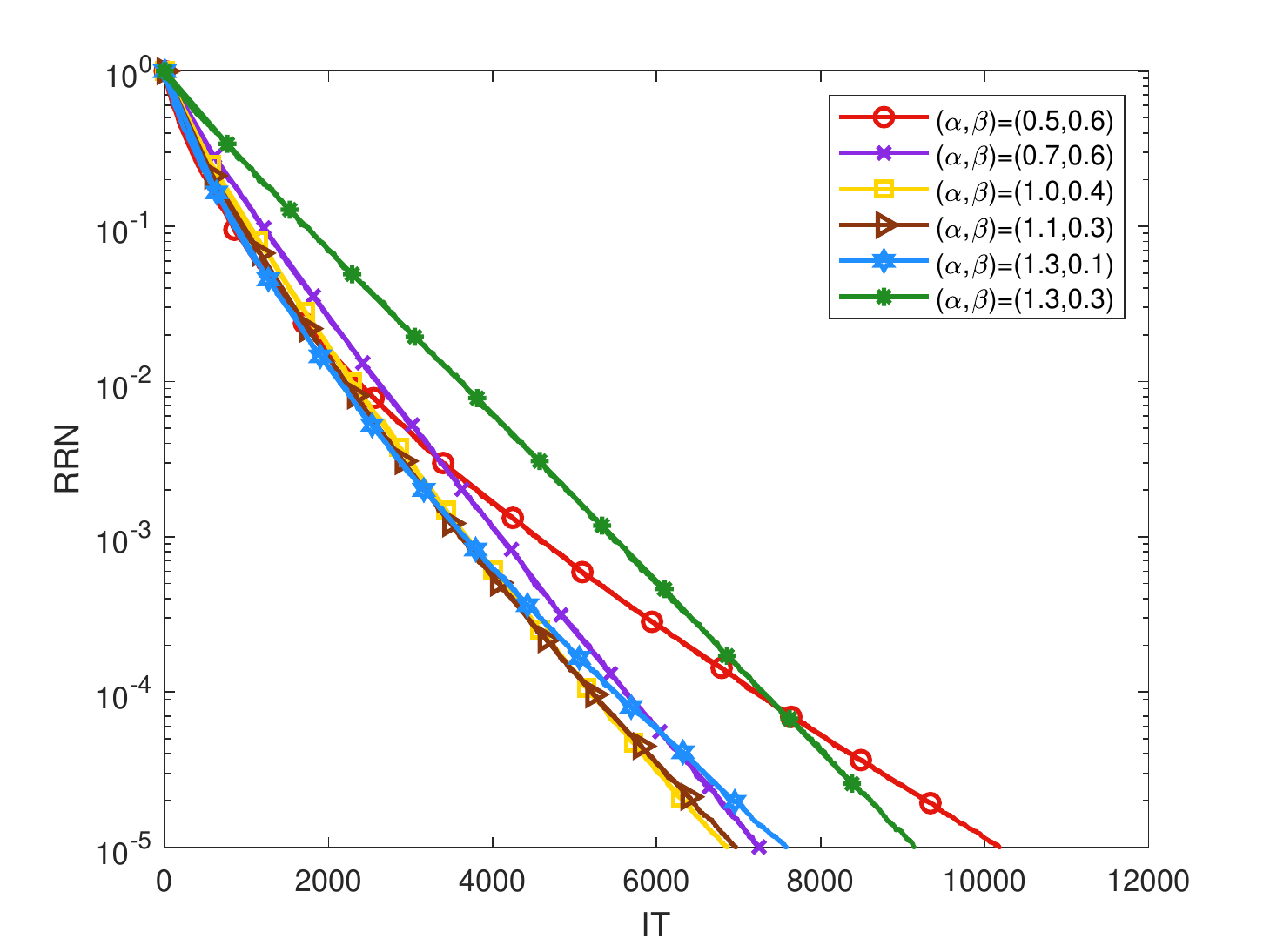}}
\caption{The convergence behaviors of RRN versus IT given by the PmRGRK method for Example \ref{Example_4} with different $\alpha$ and $\beta$ when $m=400$, $n=50$, and $p=100$, and $\theta=0.5$, $0.7$, and $0.9$.}
\label{fig:PmRGRK_Example_4-alpha+beta}
\end{figure}

\begin{figure}[!htb]
    \subfigure[$\theta=0.5$]{
		\includegraphics[width=0.48\textwidth]{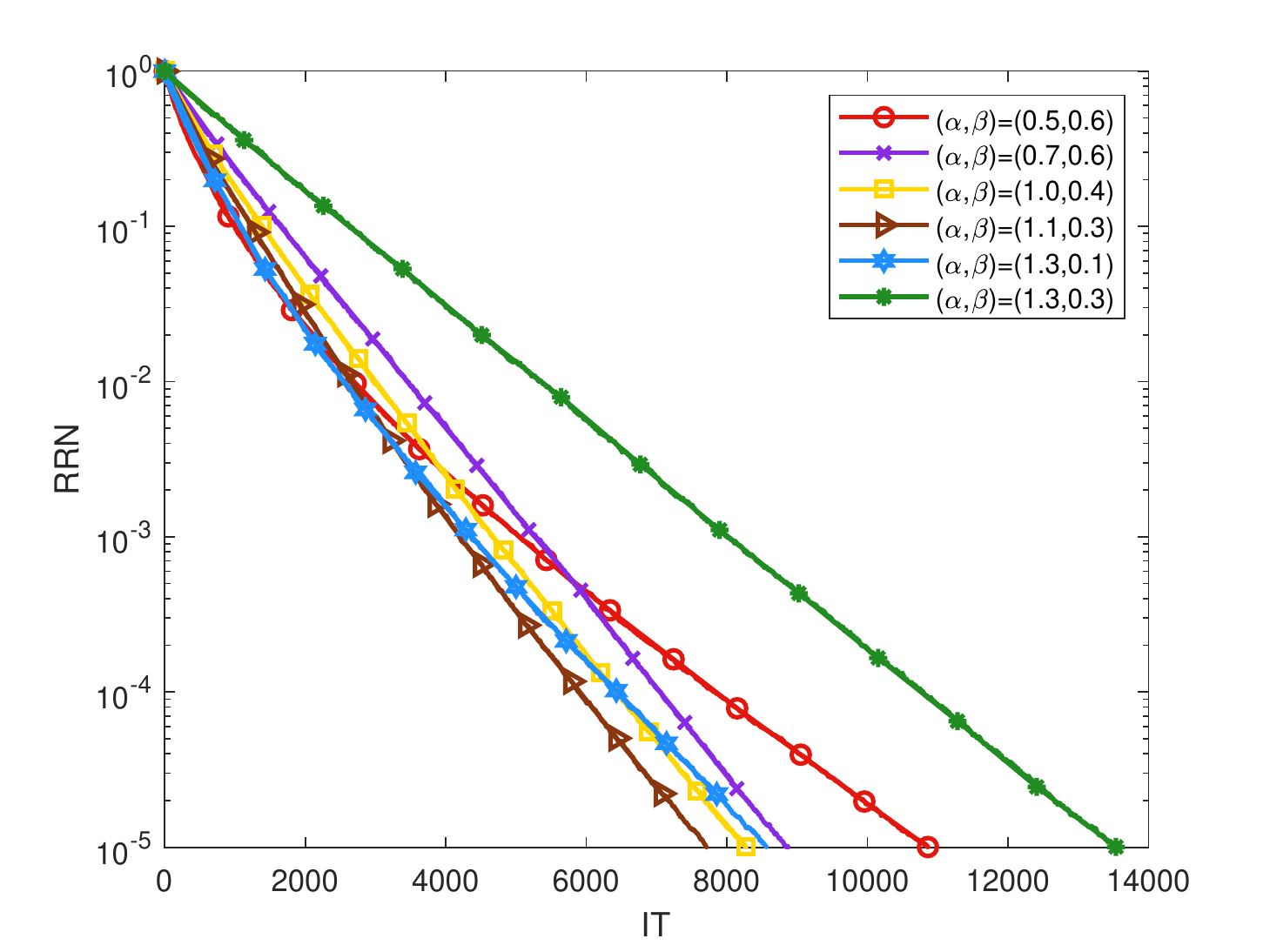}}
    \subfigure[$\theta=0.7$]{
		\includegraphics[width=0.48\textwidth]{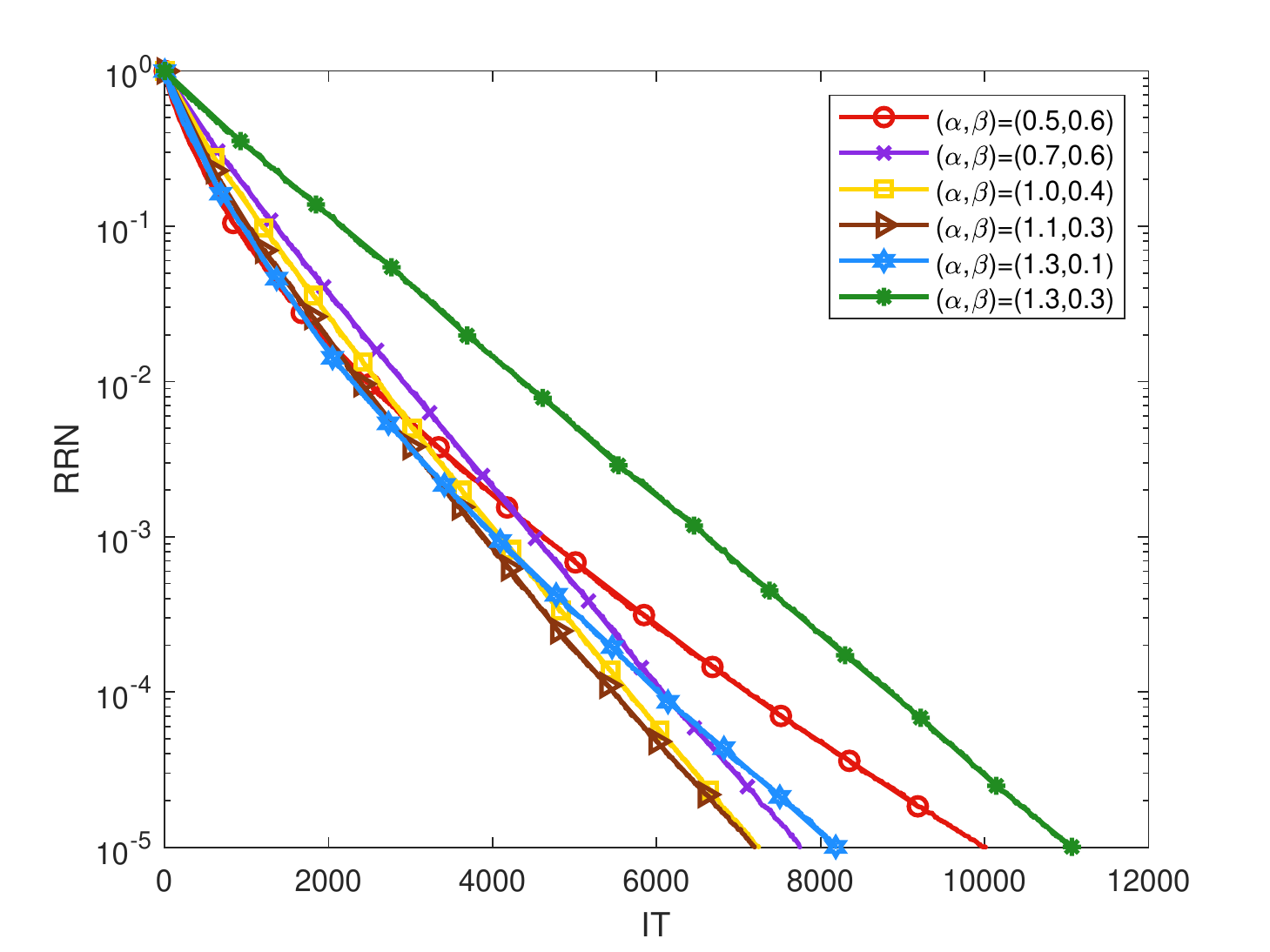}}
    \subfigure[$\theta=0.9$]{
		\includegraphics[width=0.48\textwidth]{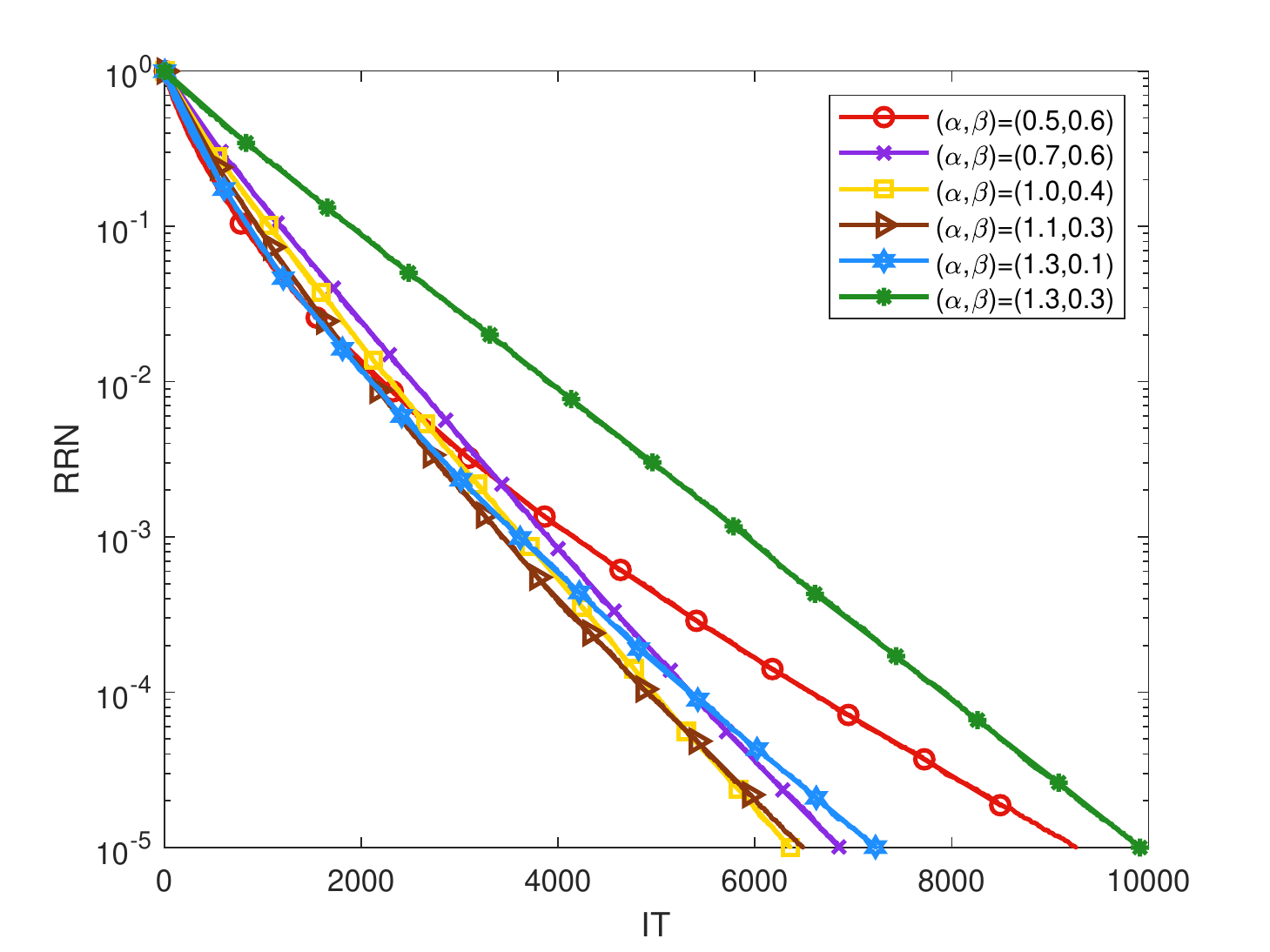}}
\caption{The convergence behaviors of RRN versus IT given by the NmRGRK method for Example \ref{Example_4} with different $\alpha$ and $\beta$ when $m=400$, $n=50$, and $p=100$, and $\theta=0.5$, $0.7$, and $0.9$.}
\label{fig:NmRGRK_Example_4-alpha+beta}
\end{figure}

\subsection{Tensor product surface fitting}
In general, geometric iterative method (GIM) computes a B-spline surface according to the following steps: parameterization, knot vector generation, and control point solving \cite{14DL, 02Farin}. A lot of work has been done on the first two steps. For example, the parameters can be evaluated through chordal parameterization \cite{14DL}  or centripetal parameterization \cite{97PT} for the data points in rows or grids. The other works on parameterization, such as polygonal meshes and point clouds, were given by \cite{16CHL, 02LPRM}. Knot vector generation plays an important role in B-spline surface fitting, since it decides the number of control points \cite{22JWH}. After that, the surface can be obtained by control point solving. 

We consider to fit the ordered data point set $ \{Q_{i,j} \in \mathbb{R}^3 : i\in[m], j\in[p] \}$ sampled uniformly from the following two surfaces, whose coordinates are given by
\begin{align*}
{\rm Surface~1:}
\left \{
\begin{array}{l}
x = -2 t \cos(s) + 2 \cos(s) / t - 2 t^3 \cos(3 s) / 3\vspace{1ex}\\
y = 6t\sin(s) - 2 \sin(s) /t - 2 t^3 \sin(3 s) / 3\vspace{1ex}\\
z = 4 \log(t)
\end{array}
\right.
\end{align*}
for $0.5 \leq t \leq 1,~0 \leq s \leq 2 \pi$ and
\begin{align*}
{\rm Surface~2:}
\left \{
\begin{array}{l}
x = (2 + \cos(t)) (s/3 - \sin(s))           \vspace{0.75ex}\\
y = (2 + \cos(t - 2 \pi / 3)) (\cos(s) - 1) \vspace{0.75ex}\\
z = (2 + \cos(t + 2 \pi / 3)) (\cos(s) - 1)
\end{array}
\right.
\end{align*}
for $-\pi \leq t \leq  \pi,~-2\pi \leq s \leq 2\pi$, respectively; see also \verb"http://paulbourke.net/geometry/".

The B-spline tensor product fitting surface can be represented as
\begin{align}\label{eq:surface_S(x)}
  S^{(k)}(x,y) = \sum_{h=1}^n\sum_{t=1}^n P_{h,t} \phi_{h,t}(x,y),
\end{align}
where $\phi_{h,t}(x,y)=\phi_{h}(x)\phi_{t}(y)$ are the products of two B-spline bases with the uniform knot vectors, $(x,y)$ is determined in $\Omega \subset \mathbb{R}^2$, and $P_{h,t}$ are the control coefficients for $h$, $t=1,2,\cdots,n$. The new control coefficients of GIM are updated by
\begin{align*}
 P_{h,t}^{(k+1)} =  P_{h,t}^{(k)} +   \Delta_{h,t}^{(k)},
\end{align*}
where $\Delta_{h,t}^{(k)} \in \mathbb{R}^3$ is called the adjust vector and computed by $P_{h,t}^{(k)}$ and $Q_{i,j}$.

Let the $x$-, $y$-, and $z$-coordinates of control coefficient $P_{h,t}^{(k)}$ (resp., adjust vector $\Delta_{h,t}^{(k)}$) be respectively stored in the matrices $P^{(k,1)}$, $P^{(k,2)}$, and $P^{(k,3)}$  (resp., $\Delta^{(k,1)}$, $\Delta^{(k,2)}$, and $\Delta^{(k,3)}$). From algebraic aspects, the GIM iterative processes,
\begin{align*}
   P^{(k+1,1)} = P^{(k,1)} + \Delta^{(k,1)},~
   P^{(k+1,2)} = P^{(k,2)} + \Delta^{(k,2)},~{\rm and}~~
   P^{(k+1,3)} = P^{(k,3)} + \Delta^{(k,3)},
\end{align*}
are equal to solving three matrix equation, where the coefficient matrices are the collocation matrices of the B-spline bases on a parameter sequence and knot vector. Therefore, the ME-RGRK, PmRGRK, and NmRGRK methods are suitable for tensor product surface fitting.

The implementation details of ME-RGRK, PmRGRK, and NmRGRK tensor product surface fittings are presented as follows. We first arrange the data points into a three-order tensor $Q = \left[Q_1;~Q_2;~Q_3 \right] \in \mathbb{R}^{m \times p \times 3}$ and input the initial values, including
two collocation matrices $A$ and $B^T$,
initial vectors $P^{(0,1)}\in \mathbb{R}^{n\times n}$ (resp., $P^{(0,2)}$, $P^{(0,3)}$),
and the right-hand side $Q_1$ (resp., $Q_2$, $Q_3$), where the collocation matrices and data points are generated by Example \ref{Ex:CAGD-AXB=C}. Then, we compute the next control point matrix $P^{(k,1)}$ (resp., $P^{(k,2)}$, $P^{(k,3)}$) using the update rules in Algorithms \ref{alg:ME-RGRK}, \ref{alg:PmRGRK}, and \ref{alg:NmRGRK}. As a result, the approximate surface is formulated according to formula \eqref{eq:surface_S(x)}.

\begin{example}\label{Ex:CAGD-AXB=C}
In this example, the coefficient matrices in \eqref{AXB=C} are from the tensor product surface fitting. As other researchers do, we first assign two parameter sequences $\nu_1$ and $\nu_2$, and two knot vectors $\mu_1$ and $\mu_2$ of cubic B-spline basis, whose formulations can be  respectively referred by equations (9.5) and (9.69) in the book \cite{97PT}. Then, we obtain the collocation matrices using the MATLAB built-in function as $A={\sf spcol}(\nu_1, 4, \mu_1)$ and $B^T={\sf spcol}(\nu_2, 4, \mu_2)$.
\end{example}

The PmRGRK and NmRGRK methods can be started with arbitrary initial control points in the column space of the collocation matrix. A suitable and efficient choice is to set the initial control points to be $P^{(0)}_{ij} =  Q_{f_1(i), f_2(j)}$, where
\begin{align*}
f_1(i) =
\left \{
\begin{array}{ll}
1, &  i=1          \vspace{0.75ex}\\
\lfloor  mi/n  \rfloor, & i\neq 1
\end{array}
\right.
\quad {\rm and} \quad
f_2(j) =
\left \{
\begin{array}{ll}
1, &  j=1          \vspace{0.75ex}\\
\lfloor  pj/n  \rfloor, & j\neq 1
\end{array}
\right.
\end{align*}
for $i,j\in [n]$ with the notation $\lfloor \cdot \rfloor$ being the greatest integer function; see, e.g., \cite{14DL,02Farin, 97PT}. At the $k$th iteration, the relative  residual norm is defined by RRN = $E_k/E_0$ for $k=0,1,2,\cdots$, where $E_k$ is computed by
\begin{align*}
 E_k = \bF{ Q_1 -  AP^{(k,1)}B}+\bF{ Q_2 -  AP^{(k,2)}B}+\bF{ Q_3 -  AP^{(k,3)}B}.
\end{align*}
The computation is terminated once RRN is less than $5 \times 10^{-4}$.

Next, we utilize the ME-RGRK, PmRGRK, and NmRGRK methods to fit $4\times 10^4$ ($m=400$ and $p=100$) data points, as shown in Figure \ref{fig:PmRGRK+Ex5_Surface_initial}. These experiments are realized by using cubic B-spline bases with $n=100$. The relaxation parameter is set to be $\theta=0.9$. The computational results of RRN, IT, and CPU are reported in Table \ref{tab:PmRGRK+NmRGRK:Exp5+m400p100n100}. We can read that the PmRGRK and NmRGRK methods always successfully compute an approximate solution for the matrix equation \eqref{AXB=C} arising from Example \ref{Ex:CAGD-AXB=C}, but the ME-RGRK method fails due to the number of the iteration steps exceeding $1\times 10^{5}$. Hence, the PmRGRK and NmRGRK methods significantly outperform the ME-RGRK method in terms of iteration count. The fitting surfaces constructed by the PmRGRK and NmRGRK methods are shown in Figures \ref{PmRGRK+NmRGRK-fig:VerrillMinimal} and \ref{PmRGRK+NmRGRK-fig:BentHorns}.

\begin{figure}[!htb]
    \subfigure[Surface 1]{
		\includegraphics[width=0.48\textwidth]{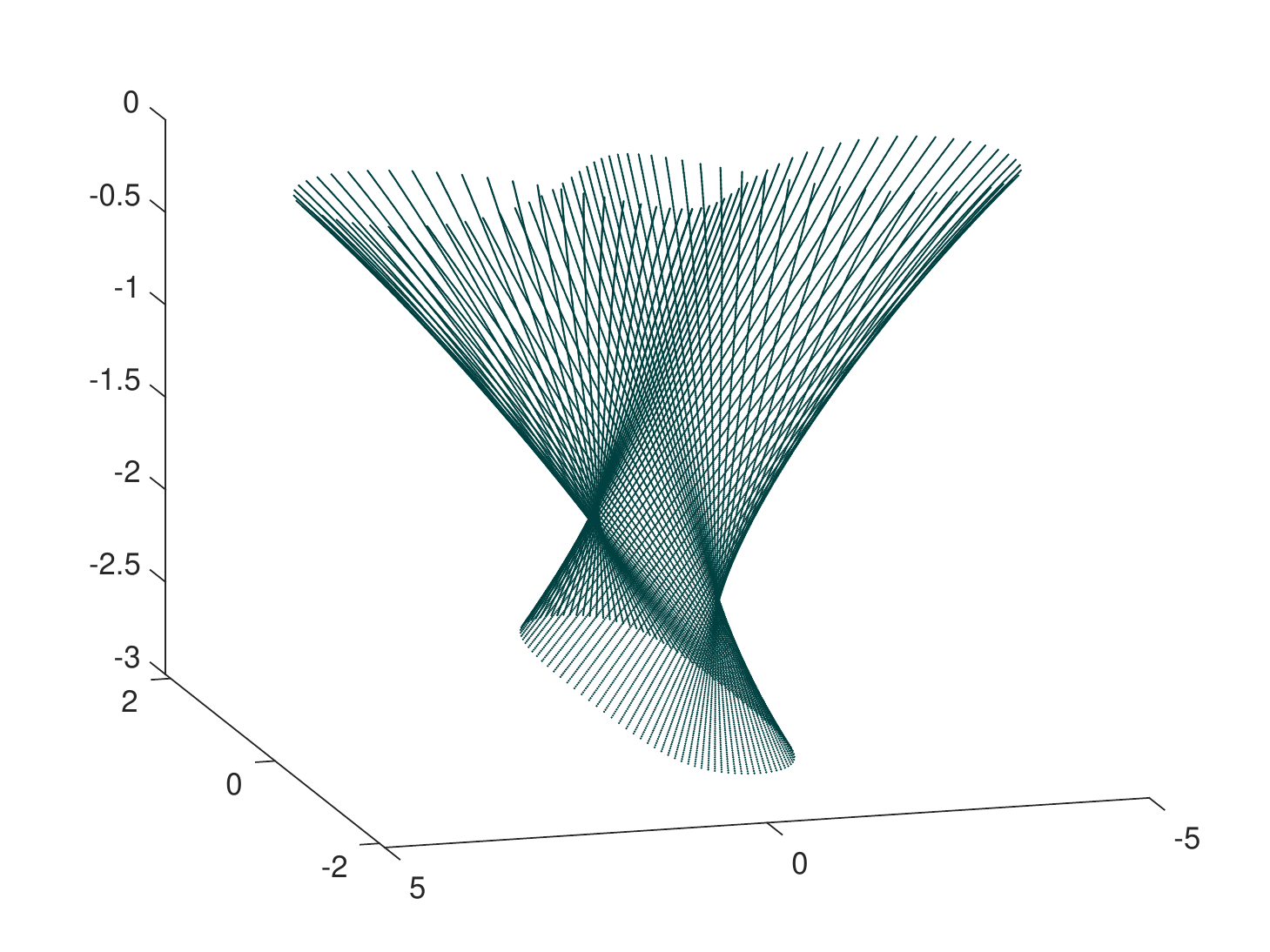}}
    \subfigure[Surface 2]{
		\includegraphics[width=0.48\textwidth]{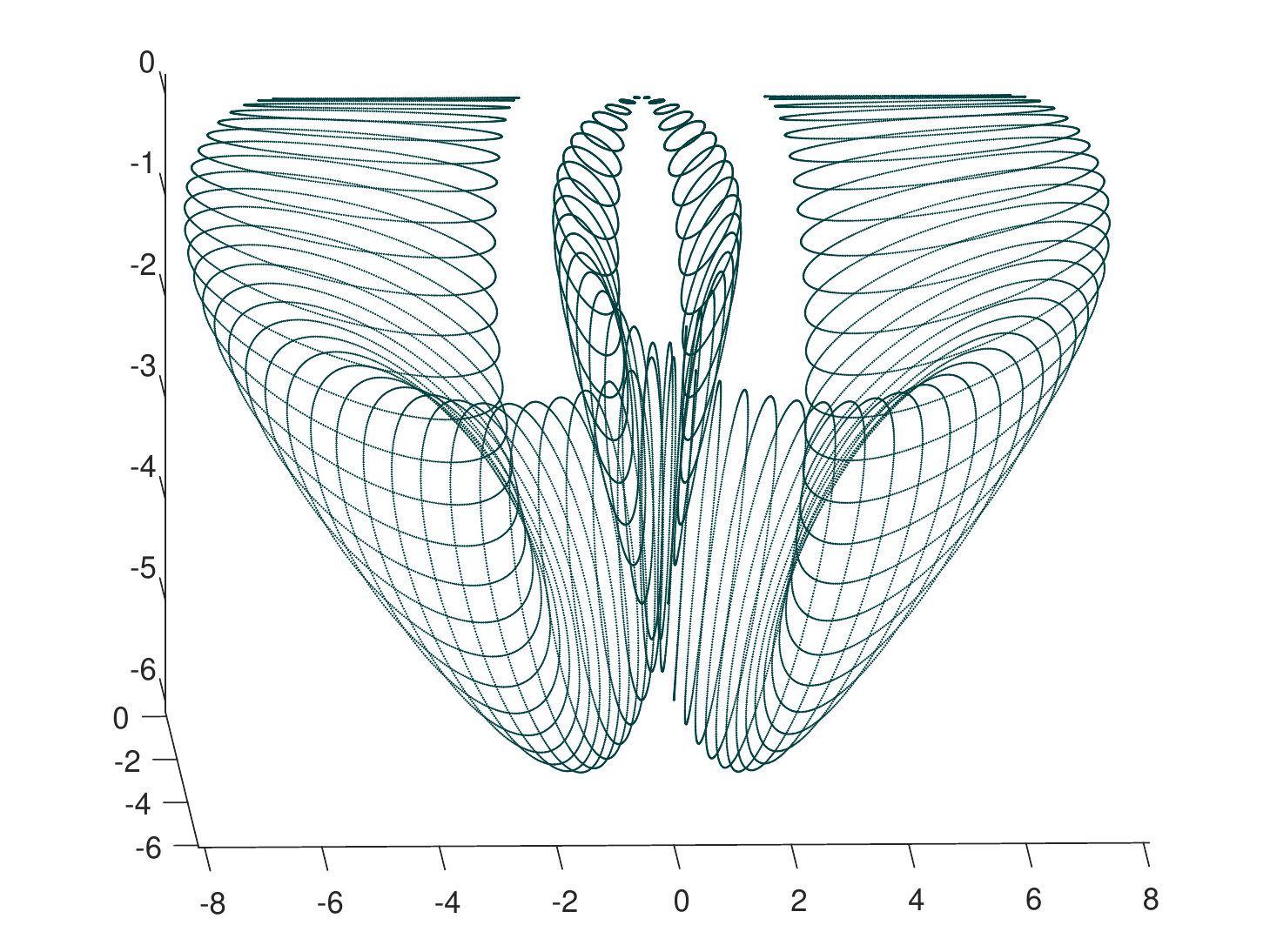}}
\caption{The initial data points to be fitted, which are sampled from Surfaces 1 (a) and 2 (b) with $m=400$ and $p=100$.}
\label{fig:PmRGRK+Ex5_Surface_initial}
\end{figure}

\begin{table}[!htb]
 \normalsize
\caption{RRN, IT, and CPU for the ME-RGRK, PmRGRK, and NmRGRK methods with $m=400$, $p=100$, and $n=100$ in Example \ref{Ex:CAGD-AXB=C}.}
\centering
\begin{tabular}{|l|c|c|c|c|}
\cline{1-5}
\multicolumn{2}{|l|}{Methods}&ME-RGRK&PmRGRK&NmRGRK\\
\cline{1-5}
\multirow{3}*{\rm Surface~1} &RRN      &$\#$             &$4.9999\times 10^{-4}$&$4.9999\times 10^{-4}$\\ \cline{2-5}
                             &IT       &$>1\times 10^{5}$&$85835$  &$72776$\\ \cline{2-5}
                             &CPU      &$\#$             &131.94   &129.59\\ \hline
\multirow{3}*{\rm Surface~2 }&RRN      &$\#$             &$4.9998\times 10^{-4}$&$4.9995\times 10^{-4}$\\ \cline{2-5}
                             &IT       &$>1\times 10^{5}$  &$85280$  &$72052$\\ \cline{2-5}
                             &CPU      &$\#$               &125.14   &122.67\\  \hline
\end{tabular}
\begin{tablenotes}
\footnotesize
\item[1.] {\it The item ' $>1\times 10^{5}$' represents that the number of iteration steps exceeds $1\times 10^{5}$. In this case, the corresponding RRN and CPU are expressed by $'\#'$.}
\end{tablenotes}
\label{tab:PmRGRK+NmRGRK:Exp5+m400p100n100}
\end{table}

\begin{figure}[!htb]
\subfigure[The PmRGRK surface]{\includegraphics[width=0.48\textwidth]{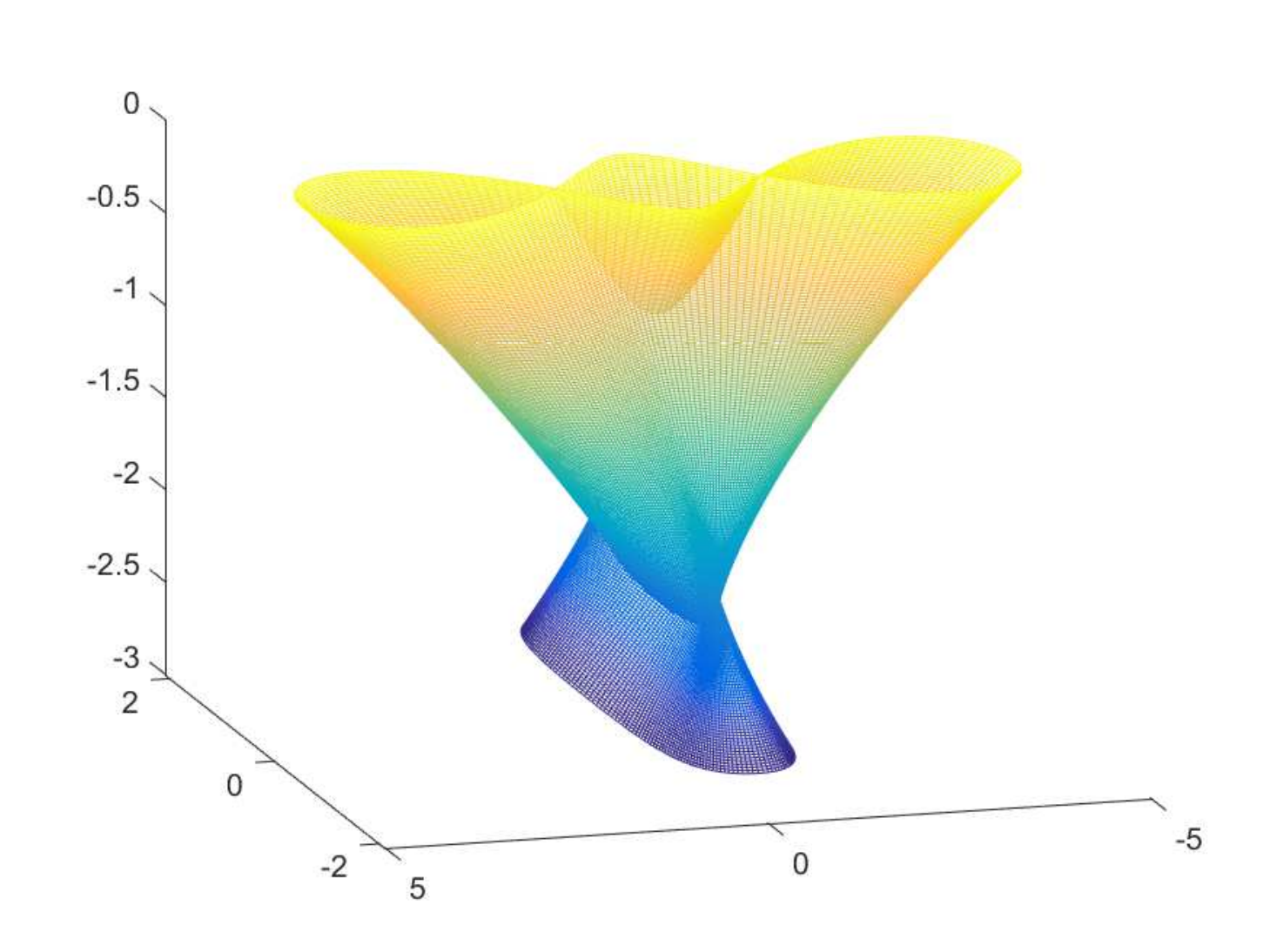}}
\subfigure[The NmRGRK surface]{\includegraphics[width=0.48\textwidth]{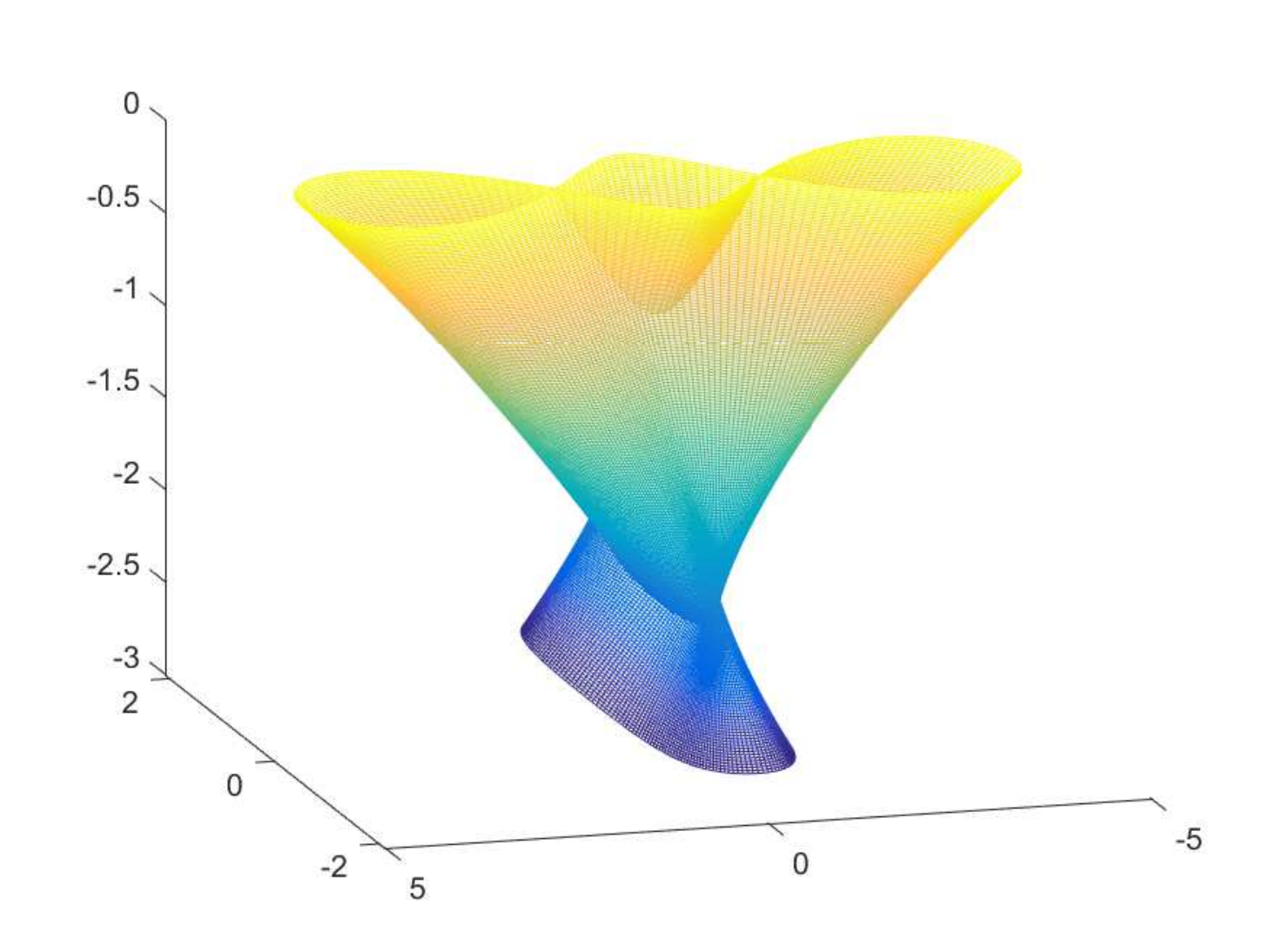}}
\caption{The first kind of fitting surfaces generated by the PmRGRK (a) and NmRGRK (b) methods for Example \ref{Ex:CAGD-AXB=C}.}
\label{PmRGRK+NmRGRK-fig:VerrillMinimal}
\end{figure}

\begin{figure}[!htb]
\subfigure[The PmRGRK surface]{\includegraphics[width=0.48\textwidth]{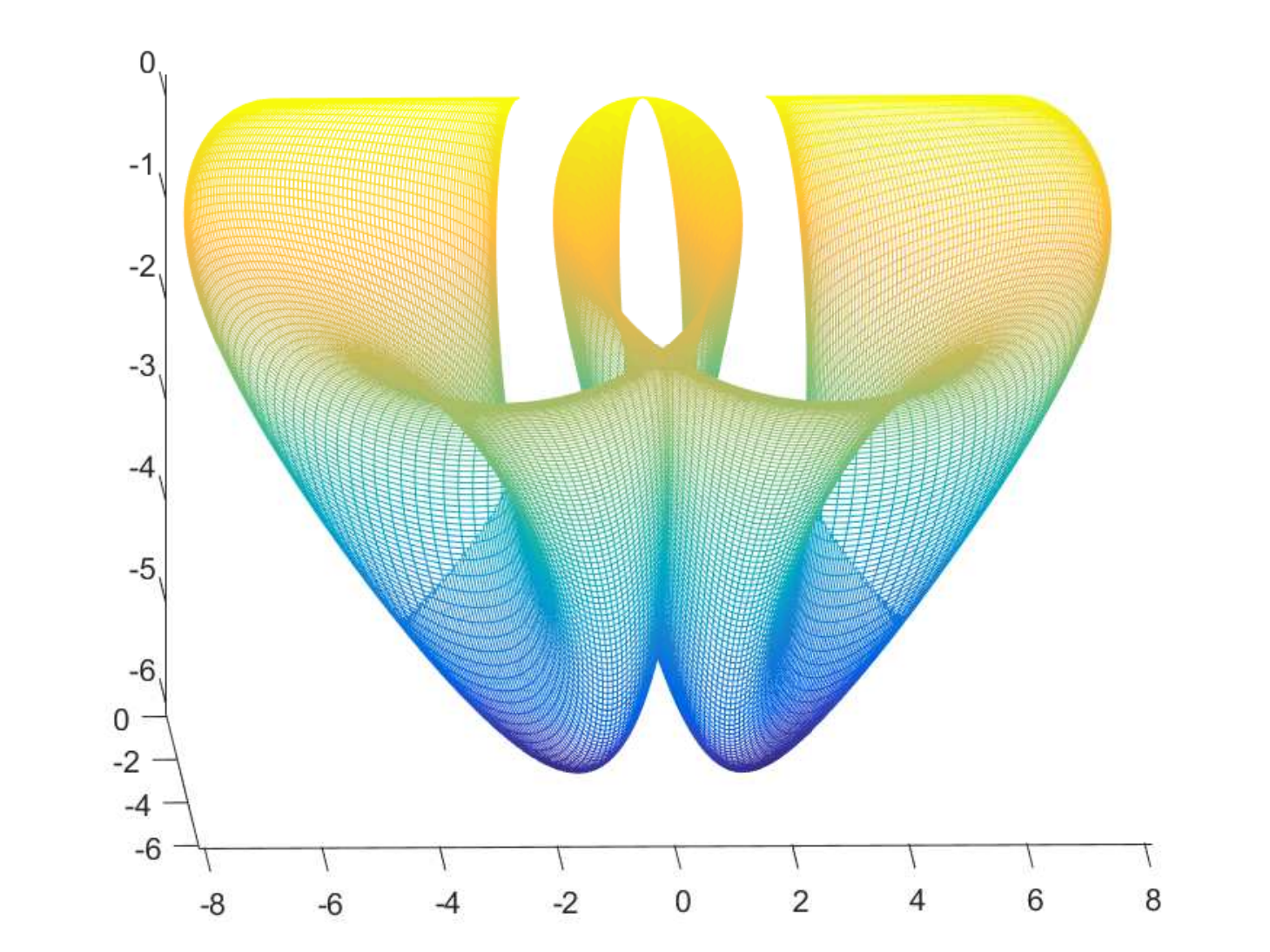}}
\subfigure[The NmRGRK surface]{\includegraphics[width=0.48\textwidth]{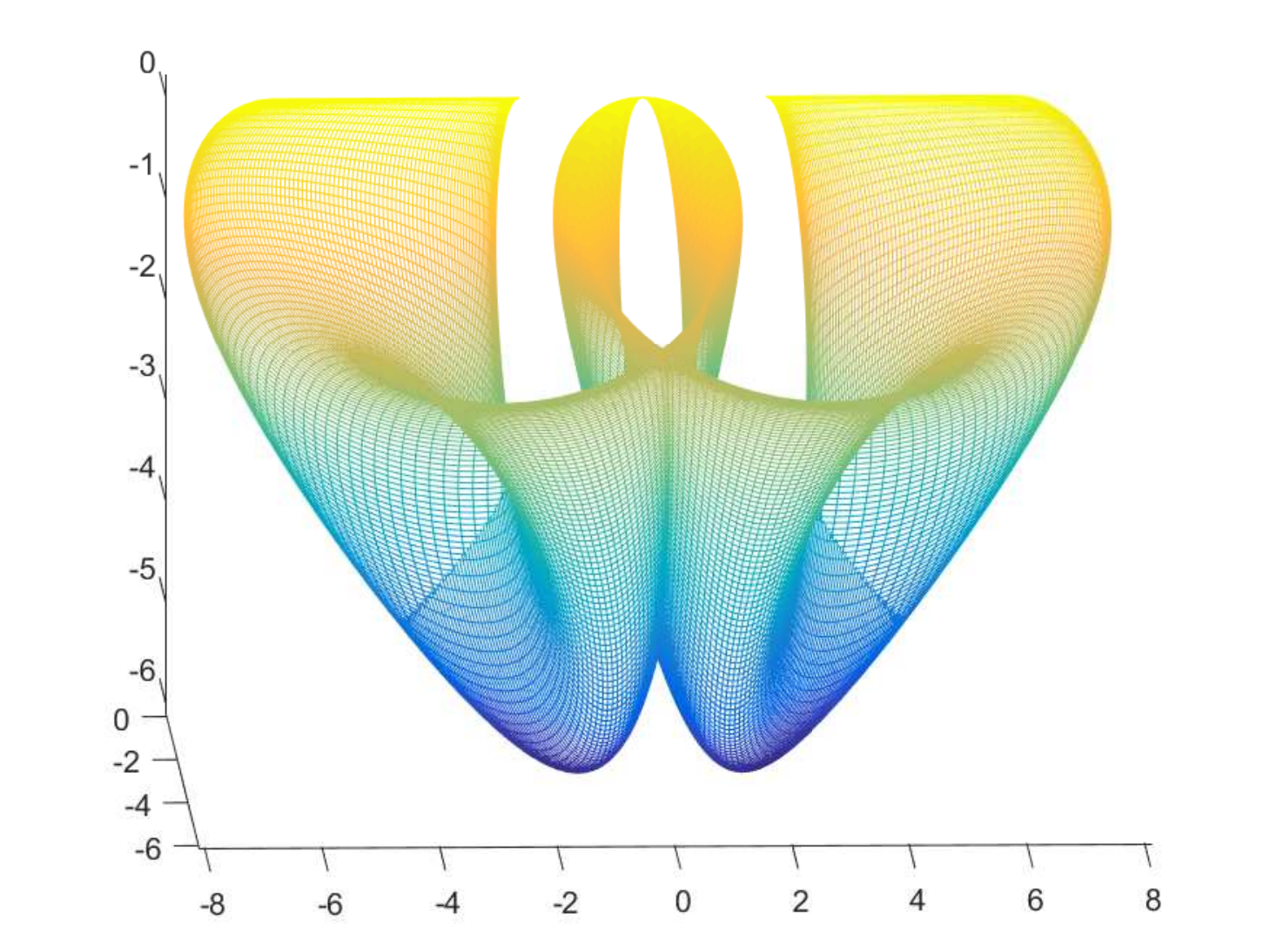}}
\caption{The second kind of fitting surfaces generated by the PmRGRK (a) and NmRGRK (b) methods for Example \ref{Ex:CAGD-AXB=C}.}
\label{PmRGRK+NmRGRK-fig:BentHorns}
\end{figure}

\section{Conclusions}\label{sec:PmRGRK+NmRGRK_conclud}

For iteratively computing the minimum Frobenius-norm least-squares solution of a consistent matrix equation, the ME-RGRK method \cite{22WLZ} was proposed by combining the matrix version of Kaczmarz iteration scheme with the relaxed greedy randomized index selection strategy. To further accelerate the convergence rate of the ME-RGRK method, in this work, we utilize Polyak's  and Nesterov's  momentum acceleration techniques and present the PmRGRK and NmRGRK methods. Convergence theories have been developed. The corresponding computational complexity analyses are also given. Some numerical examples, where the coefficient matrix is obtained from synthetic data and tensor product surface fitting, are presented to demonstrate their numerical advantage over the ME-RGRK method in terms of iteration counts and computing times. Numerical results illustrate that the PmRGRK and NmRGRK methods are competing Kaczmarz variants for solving the consistent matrix equation \eqref{AXB=C}.

Finally, we would like to make some comments on the possible extensions of our methods.

(1) We choose to use the fixed step-size and momentum parameter at each iteration for the PmRGRK and NmRGRK methods, but it is possible to extend the method to have varying ones. For example, the asynchronous PmRGRK iteration is given by
 \begin{align*}
 X^{(k+1)} = X^{(k)} +  \alpha_k \frac{C_{i,j}-a_i^T X^{(k)} b_j }{\norm{a_i}^2  \norm{b_j}^2} a_i  b_j^T
 + \beta_k ( X^{(k)} - X^{(k-1)} ),
\end{align*}
for $k=1,2,\cdots$ with $i\in [m]$ and $j\in [p]$, where $\alpha_k$ and $\beta_k$ are the adaptive step-size and momentum parameter, respectively. In particular, the theoretically upper bounds for these two parameters derived in Theorems \ref{thm:CA-PmRGRK} and \ref{thm:CA-NmRGRK} are difficult to estimate a priori. The promising adaptive selection strategy can be achieved by the information available at the beginning as analyzed in, e.g., \cite{16LW}.

(2) Variants of the Kaczmarz method that make use of more than a single row index pair at each iteration, are often referred to as block methods. At the $k$th iteration, a block of row index pair $(i,j)$ from $\Delta_k$ is selected. Then, the projections of $X^{(k)}$ onto each row index pair in $\Delta_k$ may be computed and averaged, such as in a weighted fashion. The asynchronous $\alpha_k$ and $\beta_k$ potentially dependent on the iteration are used. Taking the Polyak's  momentum as a example, the resulting update rule is  given by
\begin{align*}
 X^{(k+1)} = X^{(k)} +  \alpha_k\sum_{(i,j)\in \Delta_k} \omega_{i,j}^{(k)}\frac{C_{i,j}-a_{i}^T X^{(k)} b_{j} }{\norm{a_{i}}^2  \norm{b_{j}}^2} a_{i}  b_{j}^T
 + \beta_k ( X^{(k)} - X^{(k-1)} ),
\end{align*}
where the weights $\omega_{i,j}^{(k)}$ satisfy $\sum_{(i,j)\in \Delta_k} \omega_{i,j}^{(k)}=1$. The randomized averaged block method with no momentum can be found in \cite{22NZ}.


\section*{Acknowledgment}
The authors express their appreciation for supports provided by the National Natural Science Foundation of China under grants 12201651 and 52263002. The third author was also supported by the Natural Science Foundation of Shanxi Province under grant  20210302123480. These supports are gratefully acknowledged.

\end{document}